\documentclass[dvipsnames]{article}

\usepackage[latin1, utf8]{inputenc}
\usepackage[T1]{fontenc}
\usepackage[english]{babel}
\usepackage{enumerate}
\usepackage{todonotes}
\usepackage{hyperref}
\usepackage{authblk}
\usepackage{amsmath}
\usepackage{amssymb}
\usepackage{amsthm}
\usepackage{mathrsfs}
\usepackage{xkeyval}
\usepackage{stmaryrd}
\usepackage{tocbibind}
\usepackage{comment}
\usepackage{caption}
\usepackage{subcaption}
\usepackage{tikz-cd} 

\usepackage[left=2.5cm,top=3cm,right=2.5cm,bottom=3cm,bindingoffset=0.5cm]{geometry}
\newcommand{\R}{\mathbb{R}}
\newcommand{\N}{\mathbb{N}}

\newcommand{\U}{\mathcal{U}}
\newcommand{\cone}{\mathrm{cone}}
\newcommand{\co}{\mathrm{conv}}
\newcommand{\sing}{\mathrm{sing}}
\newcommand{\adhco}{\overline{\mathrm{conv}}}
\newcommand{\tar}{y_f}
\newcommand{\adjtar}{p_f}

\newcommand{\UU}{\mathbf{U}}

\newcommand{\scl}[2]{\left\langle #1, #2\right\rangle}

\newcommand{\ext}{\mathrm{ext}}
\newcommand{\dom}{\mathrm{dom}}

\newcommand{\ie}{\textit{i.e., }}

\DeclareMathOperator*{\argmax}{arg\,max}

\newcommand {\e}  {\varepsilon}
\newcommand{\vertiii}[1]{{\vert\kern-0.25ex\vert\kern-0.25ex\vert #1 
    \vert\kern-0.25ex\vert\kern-0.25ex\vert}}

\newtheorem{prpstn}{Proposition}[section]
\newtheorem{lmm}{Lemma}[section]
\newtheorem{thrm}{Theorem}[section]

\newtheorem{crllr}{Corollary}[section]
\newtheorem{rmrk}{Remark}[section]

\everymath{\displaystyle}
\title{Constructive reachability for linear control problems under conic constraints}
\author[a]{Camille Pouchol}
\author[b,d]{Emmanuel Trélat}
\author[c]{Christophe Zhang}
\affil[a]{Laboratoire MAP5 UMR 8145,
Université Paris Cité, 75006 Paris, France. Email address: camille.pouchol@u-paris.fr}
\affil[b]{Sorbonne Universit\'{e}, Universit\'{e} de Paris, CNRS, Laboratoire Jacques-Louis Lions, 75005 Paris, France. Email address: emmanuel.trelat@sorbonne-universite.fr}
\affil[c]{Université de Lorraine, CNRS, Inria, IECL, F-54000 Nancy, France. Email address: christophe.zhang@polytechnique.org}
\affil[d]{CAGE, INRIA, Paris, France.}
\date{\empty}

\begin{document}
\maketitle

\begin{abstract}

Motivated by applications requiring sparse or nonnegative controls, we investigate reachability properties of linear infinite-dimensional control problems under conic constraints. 
Relaxing the problem to convex constraints if the initial cone is not already convex, we provide a constructive approach based on minimising a properly defined dual functional, which covers both the approximate and exact reachability problems. Our main results heavily rely on convex analysis, Fenchel duality and the Fenchel-Rockafellar theorem. 
As a byproduct, we uncover new sufficient conditions for approximate and exact reachability under convex conic constraints. We also prove that these conditions are in fact necessary.
When the constraints are nonconvex, our method leads to sufficient conditions ensuring that the constructed controls fulfill the original constraints, which is in the flavour of bang-bang type properties.
We show that our approach encompasses and generalises several works, and we obtain new results for different types of conic constraints and control systems.

\end{abstract}

\tableofcontents

\section{Introduction}
\subsection{Control problem and motivations}
We let $X$ and $U$ be two Hilbert spaces, $T>0$ be a final time, and denote $E := L^2(0,T;U)$.

We are given an unbounded operator $A : \mathcal{D}(A) \rightarrow X$ generating a $C_0$ semigroup on $X$, denoted $(S_t)_{t \geq 0}$, and $B \in L(U,X)$ a control operator. We then consider the linear control problem
\begin{equation}
\label{control}
    \dot y = Ay + Bu.
\end{equation}

For a constraint set $P \subset U$, a given initial condition $y_0\in X$ and a target $\tar\in X$, we say that 
\begin{itemize}
\item 
\textit{$\tar$ is approximately reachable from $y_0$ in time $T$
under the constraints $P$}
if for all $\e>0$, there exists $u_\e \in E$ such that for a.e. $t \in (0,T)$, $u_\e(t) \in P$ and the corresponding solution to~\eqref{control} with $y(0) = y_0$ and control $u_\e$ satisfies $\|y(T) - \tar\|_X \leq \e$. 
\item \textit{$\tar$ is exactly reachable from $y_0$ in time $T$
under the constraints $P$} if
there exists $u \in E$ such that for a.e. $t \in (0,T)$, $u(t) \in P$ and the corresponding solution to~\eqref{control} with $y(0) = y_0$ and control $u$ satisfies $y(T) = \tar$.
\end{itemize} 
If there are no constraints, \textit{i.e.} $P = U$, we will simply write that $\tar$ is approximately (resp. exactly) reachable from $y_0$ in time $T$.

We are interested in the (constrained) reachability problem when the constraint set $P \subset U$ is a \textbf{cone}, which we will call the constraint cone. In this unbounded setting, it is relevant to distinguish between approximate and exact reachability as these two notions need not coincide, even when $P$ is closed and in finite dimension. An example of this phenomenon in dimension $2$ is provided in Appendix~\ref{app-counterexample}.

More precisely, we aim at 
\begin{itemize}
\item deriving \textbf{necessary and sufficient conditions} for approximate and exact reachability,
\item
developing \textbf{constructive} approaches for the design of controls achieving reachability.
\end{itemize}

The motivation for unbounded conic constraints mainly comes from two main types of constraints, both of interest for applications.
The first type is that of~\textbf{nonnegative} constraints, when $X$ is a finite-dimensional space or a functional space (such as $L^2$). These are convex conic constraints.
The second type is concerned with \textbf{sparsity} constraints. Roughly speaking, these require that, at all times, only a few controls be active. These constraints are not convex and hence prove to be more challenging. 

Moreover, as we work with a fixed control time $T$, it is more relevant to consider unbounded constraint sets such as cones, on which optimal control problems with natural quadratic costs can be formulated. On the other hand, bounded constraint sets appear more naturally in minimal time control problems. 

\subsection{State of the art}

\paragraph{Unconstrained reachability and controllability.}
The derivation of necessary and sufficient conditions for (unconstrained) reachability associated to linear problems can be traced back to the works of Kalman~\cite{kalman1960}, with a focus on controllability, i.e., reachability results that are independent of the initial state $y_0$ and the target $\tar$ (and, possibly, of the time $T>0$). 

Since then, many such controllability conditions that properly generalise to infinite-dimensional settings have been developed, such as Hautus-type conditions \cite{hautus1970stabilization, fattorini1966complete}, unique continuation properties or observability inequalities \cite{tucsnak2009observation, CoronBook}.

In terms of constructive approaches, which for fixed values of $y_0, \tar, T$ should provide a control achieving the target, the so-called Hilbert Uniqueness Method (HUM) developed by Lions has become the method of choice~\cite{lions-1988, Lions-1992}. It is based on minimising a suitably defined functional, and its properties are intimately related to observability inequalities. Lions' variational technique has in turn inspired works that propose ad hoc functionals for specific constrained control problems~\cite{zuazua2010switching, Kunisch-Wang-2013, Ervedoza_2020}, or develop sound discretisation methods to derive the corresponding optimal control~\cite{Boyer2013}.


\paragraph{Reachability and controllability under constraints.}

The problem of constrained control of finite-dimensional linear autonomous systems of the form $\dot{x}=Ax+Bu$ has extensively been studied. The seminal paper \cite{brammer1972controllability} provides a general spectral condition on the pair $(A,B)$ to ensure constrained null-controllability, under the hypothesis that this pair is controllable. More recently, the article \cite{Loheac2021} studied the controllability of linear autonomous systems with positive controls, under the assumption that the Kalman controllability condition holds. In this controllable framework, the authors show that the positivity constraint induces positive minimal control times, and obtain constructive controls through a variational approach.

In infinite dimension, there is no equivalent for the Kalman controllability criterion, and other approaches have been developed to study constrained control problems. The author of \cite{Berrahmoune2014} develops a variational method which yields necessary and sufficient observability-type conditions for the constrained exact controllability of autonomous linear systems in Hilbert spaces.

Null-controllability conditions for bounded control sets were established in \cite{ahmed_finite-time_1985} and recovered in \cite{Berrahmoune2019}, with a focus on conservative systems. For the particular case of controls lying in balls, and focusing on parabolic type equations, \cite{berrahmoune_variational_2020} gives necessary and sufficient conditions for null-controllability. 

In all of the above works, the authors develop a variational approach akin to ours, relying on duality in the convex optimisation sense, and obtain constructive controls steering the system to the desired target states. However, it is concerned with general (null) controllability, which leads to strong conditions.

It is worth stressing that, in the presence of constraints, the right notion of controllability becomes unclear. Nonnegative constraints typically may lead to obstructions to controllability: for instance, one has $\tar \geq S_T y_0$ for the heat equation with internal nonnegative constraints, by the parabolic comparison principle \cite{Loheac2017, Nous}. Hence, deriving controllability results (i.e., uniform results with respect to $y_0$, $\tar$ or both) typically leads to restricting the notion of controllability to a well-defined subset of initial and final targets, as is done in~\cite{Nous}.

For general constraints and without specific structural assumptions about the control, the notion of reachability hence becomes more flexible and natural.

\paragraph{Reachability and controllability with conic constraints.}
More recently, unbounded constrained controllability or reachability has been the subject of revived interest, motivated by applications where the control should be nonnegative, or more generally, when the constraints on the control are unilateral~\cite{Loheac2017, Loheac2021}. 


Another line of research is that of sparse controls \cite{zuazua2010switching, Nous}. The former article~\cite{zuazua2010switching} is focused on approximate and exact reachability for finite-dimensional problems with $m$ controls, with the constraint that at all times, only one control should active. The latter article~\cite{Nous} is concerned with parabolic equations with internal nonnegative controls, and a specific sparsity constraint.

These two works rely on the analysis of a properly defined optimisation problem, through a fine study of a corresponding Fenchel dual problem, in the spirit of the HUM method. 
Both works are also based on the idea of relaxation which consists of two steps. First, one derives controls within the set of relaxed constraints (obtained by computing the closed convex hull of constraints). Second, one establishes a bang-bang type property that the obtained controls actually take values in the original constraint set.

\paragraph{Main contribution.} 
The above literature lacks a general framework to investigate constructive reachability in the relevant setting of conic constraints.

Our approach bridges this gap; it subsumes the two works~\cite{zuazua2010switching, Nous} as well as the HUM method, by providing a general recipe for constructive reachability under conic constraints. It accommodates both approximate and exact reachability, thereby yielding sufficient and necessary conditions in the case of convex constraints. 
The underlying relaxation approach is associated to general sufficient conditions under which optimal controls are bang-bang.




\subsection{Notations}
To introduce our notations, we let $H$ be a Hilbert space. For clarity, we will use the notation $P$ for cones, $C$ for convex sets, $K$ for general sets. For basic results concerning these notions, we refer, e.g., to~\cite{bauschke-combettes}.
\subsubsection{Functions}
For a function $f : H \rightarrow \mathopen{]}-\infty, +\infty]$, we let $\dom(f):= \{x \in H, \, f(x) < +\infty\}$ be its domain, and denote $\Gamma_0(H)$ the set of functions $H \rightarrow \mathopen{]}-\infty, +\infty]$ that are convex, lower-semicontinuous, and proper (\textit{i.e.}, $\dom(f) \neq \emptyset$).

For a function $f \in \Gamma_0(H)$, we let $f^\ast$ be its convex conjugate, given by \[\forall x\in H, \quad f^\star(x) := \sup_{y \in H} \, \big(\langle x, y \rangle -f(y)\big).\]
We have $f^\ast \in \Gamma_0(H)$, and Fenchel-Moreau's theorem states that $(f^\ast)^\ast = f$ for all $f \in \Gamma_0(H)$.

For $f \in \Gamma_0(H)$ and $x \in H$, we denote \[\partial f(x):= \{p \in H, \; \forall y \in H, \, f(y) \geq f(x) +  \scl{p}{y-x}\},\] the subdifferential set of $f$ at $x$.

For a (non-empty) closed convex set $C \subset H$, we let $\delta_C$ be its indicator function, \textit{i.e.} the function given by $\delta_C(x) = 0$  if $x \in C$ and $+\infty$ otherwise. We have $\delta_C \in \Gamma_0(H)$ and we let $\sigma_C := \delta_C^\ast$ be its support function, which by definition is given by \[\forall x \in H, \quad \sigma_C(x) = \sup_{y \in C} \langle x,y\rangle.\]
Finally, we let $j_C$ be the gauge function of $C$, namely
\[\forall x \in H,\quad j_C(x) := \inf \{\alpha>0, \; x \in \alpha C\}.\]
See Appendix~\ref{app-gauge-results} for some elementary results concerning gauge functions.

\subsubsection{Sets}
For a set $K \subset H$, we let 
\begin{itemize}
    \item $\overline{K}$ and $\overline{K}^{w}$ be its (strong) closure and weak closure, respectively,
    \item $\co(K)$ be its convex hull,
    \item $\cone(K)$ be the cone generated by $K$, given by \[\cone(K) = \{\lambda p, \; p \in K,\, \lambda > 0\}.\]
\end{itemize}
We also define $\adhco(K) := \overline{\co(K)} = \overline{\co(K)}^w$. We recall the caveat that $\cone(K)$ may not be closed even if $K$ is. 

For a (non-empty) closed convex set $C\subset H$, we let 
\begin{itemize}
    \item $\ext(C)$ be the set of its extremal points,
    \item $\sing(C)$ be the set of its singular normal vectors, \textit{i.e.}, vectors $v\in H$ for which the maximum defining $\sigma_C(v)$ is reached at multiple points. 
\end{itemize}
\begin{figure}[h]
     \centering
         \includegraphics[width=0.38\textwidth]{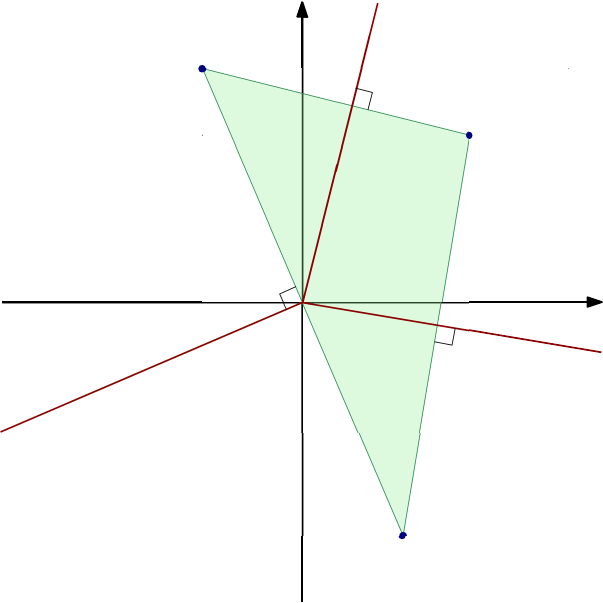}
        \caption{Example of a convex set (in green), with its extreme points (the blue dots), and the cone of its singular normal vectors (the three brown halflines).}
    \label{fig:example}
\end{figure}
We note that $\sing(C)$ is a cone (containing $0$ as soon as $C$ is not itself reduced to a singleton). For $H_2$ another Hilbert space and $B \in L(H,H_2)$, there holds $x \in \sing(B C) \iff B^\ast x \in \sing(C)$. 

We recall Milman's theorem~\cite{Rudin}: if $\adhco{(K)}$ is weakly (respectively strongly) compact, then
\[\ext(\adhco(K)) \subset \overline{K}^w \quad \text{resp. } \ext(\adhco(K)) \subset \overline{K}.\]
In particular, $\ext(\adhco(K)) \subset K$ whenever $\adhco{(K)}$ is weakly (resp. strongly) compact and $K$ weakly (resp. strongly) closed. 

Lastly, for a cone $P \subset H$, we let $P^\circ$ be its polar cone, \textit{i.e.}, $P^\circ := \{y \in H, \; \forall x \in P, \, \scl{x}{y}\leq 0\}$.

\subsection{Convex duality framework}
\subsubsection{Primal and dual optimisation problems}
Recall that we are investigating constrained reachability for conic constraints: the constraint set $P \subset U$ is a cone. Our approach will be based on generating the cone $P$ by some prescribed constraint set $\U$, which we call the \textit{generating constraint set}. 

As a result, we will be obtaining controls in $E$ such that, at a.e time $t \in (0,T)$, one can find an~\textit{amplitude} $M(t) \geq 0$ such that $u(t) = M(t) v(t)$ with $v(t) \in \U$.
In fact, our approach can also be adapted to build controls whose amplitude does not depend on $t \in (0,T)$, see our discussion in Subsection~\ref{subsec-ext-persp}.

Throughout, we will always assume that the chosen generating set $\U$ satisfies
\begin{equation}
\label{minimal-hyp}
    \text{$\U$ is bounded and $0 \in \U$.}
\end{equation}

Given a control system of the form \eqref{control}, by Duhamel's formula, one may write $y(T) = S_T y_0 + L_T u$ where $L_T \in L(E,X)$ is defined by \[L_T u:= \int_0^T S_{T-t}B u(t)\,dt.\]
Given $\e \geq 0$, finding a control $u$ taking values in $P$ steering $y_0$ to $\overline{B}(\tar,\e)$ in time $T$ is then equivalent to finding $u$ taking values in $P$ such that $L_T u \in \overline{B}(\tilde y_T, \e)$, where 
\[\tilde y_T :=  \tar - S_T y_0.\]
Of course, the case $\e=0$ is concerned with exact reachability since $\overline{B}(\tilde y_T, \e) = \{\tilde y_T\}$.

The adjoint $L_T^\ast \in L(X,E)$ of $L_T$ is given  for $\adjtar \in X$ by $L_T^\ast \adjtar(t) = B^\ast S_{T-t}^\ast \adjtar$, where $(S_t^\ast)_{t \geq 0}$ is the adjoint semigroup, generated by the adjoint operator $A^\ast$ to $A$, with domain $\mathcal{D}(A^\ast)$. In other words, one can write $L_T^\ast  \adjtar(t) = B^\ast p(t)$, where $p$ solves the equation
\begin{equation}
\label{backward}
\dot p + A^\ast p = 0, \quad p(T) = \adjtar.
\end{equation}

\paragraph{Relaxed optimisation problem.} \ Let $\U \subset U$ be a fixed generating set for $P$ i.e., such that $\cone(\U) =~P$, with $\U$ satisfying~\eqref{minimal-hyp}. We define the associated \textbf{relaxed generating constraint set}, and the \textbf{cone of relaxed constraints} it generates,
\[\U_r:=\adhco(\U), \qquad P_r:=\cone(\U_r).\]
Note that in general, $P_r \neq \adhco(P)$, since the cone generated by a closed set is not necessarily closed. This is summed up by the following diagram:

\[ \begin{tikzcd}[column sep=huge, row sep = huge]
P  & P_r  \\%
\U \arrow{r}{\adhco} \arrow{u}{\cone}& \U_r  \arrow{u}{\cone}
\end{tikzcd}
\]
This schematic highlights that our approach in building $P_r$ from $P$ is not canonical. In other words, different choices of generating sets $\U$ may lead to different cones of relaxed contraints $P_r$.

Now, we introduce the cost functional given for $u \in E$ by
\begin{equation}
\label{primal-cost}
F(u) := \tfrac{1}{2} \int_0^T j_{\U_r}^2(u(t))\,dt,
\end{equation}
which generalises the quadratic $L^2$ norm cost, and depends exclusively on $\U_r$.
We also denote the set of controls with finite cost \[L^2_{\U_r} := \{u \in E, \; t \mapsto j_{\U_r}(u(t)) \in L^2(0,T)\} = \{u \in E, \; F(u) < +\infty\}.\]

Note that if $u \in L^2_{\U_r}$, then for a.e. $t \in (0,T)$, $j_{\U_r}(u(t)) < +\infty$, hence $u(t) \in P_r$. In other words, if $u \in L^2_{\U_r}$, $u$ takes values in the cone of relaxed constraints $P_r = \cone(\U_r)$.

For a given choice of $y_0, \; \tar \in X$, $T>0$ and $\e \geq 0$, we define the following optimisation problem, in which the conic constraint $u\in P=\operatorname{cone}(\U)$ is relaxed to $u\in P_r = \operatorname{cone}(\U_r)$:  
\begin{equation}
    \label{ocp}
\inf_{u \in E, \; y(T) \in \overline{B}(\tilde y_T, \e)} F(u).
\end{equation}
By the preceding remarks, if this optimisation problem is not the trivial $+\infty$, then there exists a control $u$ taking values in the cone of relaxed constraints $\cone(\U_r)$ steering $y_0$ to $\overline{B}(\tar,\e)$ in time $T>0$.

\paragraph{Duality for the relaxed problem.} \ We show (see Proposition~\ref{compute-conj}) that $F$ is in $\Gamma_0(E)$ with, for all $p \in E$,
\begin{equation}
    \label{dual-cost}
F^\ast(p) = \tfrac{1}{2} \int_0^T \sigma_{\U_r}^2(p(t))\,dt, 
\end{equation}
and that this function takes finite values on the whole of $E$, namely $\dom(F^\ast) = E$.
Finally, for all $\e \geq 0$, we let 
\begin{equation}
\label{dual_func}
\forall \adjtar \in X, \quad J_{\e}(\adjtar) := F^\ast(L_T^\ast \adjtar) - \langle \tilde y_T, \adjtar \rangle_X + \e\|\adjtar\|_X,  
\end{equation}
and consider the associated so-called \textit{dual problem}
\begin{equation}
    \label{dual-problem}
\inf_{\adjtar \in X} J_{\e}(\adjtar).
\end{equation}
The optimisation problem~\eqref{ocp} is correspondingly called the \textit{primal problem}.

\subsubsection{Optimality conditions}
For a given choice of $\e \geq 0$, $y_0, \tar \in X$, $T>0$, the functional $J_{\e}$ may or may not admit a minimiser over~$X$.
If it does, we will see that the primal optimisation problem~\eqref{ocp} is not the trivial $+\infty$. More precisely, if $\adjtar^\star$ is a minimiser, all optimal controls must satisfy $u \in \partial F^\ast(L_T^\ast \adjtar^\star)$. Denoting $q^\star := L_T^\ast \adjtar^\star$, we will show that this corresponds to
    \begin{equation}
\label{opt_control} 
\text{for a.e. } t \in(0,T), \quad u(t) \in \sigma_{\U_r}(q^\star(t)) \, \partial \sigma_{\U_r}(q^\star(t)) = \sigma_{\U_r}(q^\star(t)) \, \argmax_{v \in \U_r}\;  \langle  q^\star(t), v\rangle_U.
    \end{equation}
   Our method is a natural extension of the HUM method (a connection on which we elaborate in Section~\ref{subsec-unconstrained-example}). It yields controls with a varying amplitude: $j_{\U_r}(u(t))$ is time dependent. 

\paragraph{Uniqueness.}    
These formulae will not always define a unique control, but we will see that at least one control satisfying~\eqref{opt_control} is an optimal control for~\eqref{ocp}. 
In fact, they will define a unique control if and only if
\begin{equation}
\tag{H}
\label{extremal}
 L_T^\ast \adjtar^\star(t)  \notin \sing( \U_r)  \cap (P_r^\circ)^c   \text{ for a.e. } t \in (0,T).
\end{equation}
This amounts to requiring $L_T^\ast \adjtar^\star$ to "avoid" the cone of singular normal vectors to $\U_r$, intersected with the complement of the polar cone $P_r^\circ$.

\paragraph{Extremality.}    
The optimal control above is a relaxation of the original constrained problem, in which the controls are required to take values in the cone of relaxed constraints $P_r$. In order to obtain controls satisfying the original constraints given by the cone $P$ upon solving the relaxed problem, a final hypothesis will be of interest: 
\begin{equation}
    \label{extremality}
    \tag{E}
    \ext(\U_r) \subset \U.
\end{equation}
Since $\U$ is assumed to be bounded, so is the closed convex set $\U_r$; hence the latter set is weakly compact. As a result, an application of Milman's theorem with either the strong or the weak topology shows that~\eqref{extremality} holds true as soon as either one of the following hypotheses holds
\begin{itemize}
    \item $\U$ is weakly closed,
    \item $\U_r$ is strongly compact and $\U$ is strongly closed.
\end{itemize}

\subsection{Main results}

Our two main theorems~\ref{thm-main-approx} and \ref{thm-main-exact} below are concerned with approximate and exact reachability, respectively. Most results are obtained by studying the dual functional~\eqref{dual_func}, whether for all $\e>0$ or for $\e=0$.
In fact, the sufficiency of our sufficient conditions for reachability~\eqref{cond-approx} and~\eqref{cond-exact} stem from that analysis, but the necessity is proved by independent arguments.   

Finally, recall the basic assumption~\eqref{minimal-hyp} about the generating constraint set, an assumption that underlies all the presented results. 
\subsubsection{Approximate reachability}
First, we consider the case of approximate reachability. In terms of constructive approaches, this corresponds to studying the dual functionals~\eqref{dual_func} for all $\e>0$.
\begin{thrm}
\label{thm-main-approx} The state $y_f$ is approximately reachable from $y_0$ in time $T>0$ under the cone of relaxed constraints $P_r = \cone(\U_r)$ if and only if  
\begin{equation}
\label{cond-approx}
  \forall \adjtar \in X, \quad   F^\ast(L_T^\ast \adjtar) = 0 \; \implies \; \scl{\tilde y_T}{\adjtar}_X \leq 0.
    \end{equation}

Now assume that the latter condition holds. Then $J_{\e}$ admits a unique minimiser $\adjtar^\star$ for any value of $\e>0$, and at least one control $u \in \partial F^\ast(L_T^\ast \adjtar^\star)$ steers the solution of~\eqref{control} from $y_0$ to the ball $\overline{B}(\tar, \e)$ in time~$T$. Furthermore,
\begin{itemize}
\item there exists a unique such control if and only if~\eqref{extremal} holds, 
\item if in addition~\eqref{extremality} holds, the unique control takes values in the original cone $P = \cone(\U)$.
\end{itemize} 
\end{thrm}


It might be surprising that condition~\eqref{cond-approx} is equivalent to an approximate reachability condition regarding $P_r$, since $F^\ast$ depends on $\U_r$. 
We provide yet another equivalent condition that indeed explicitly depends only on the cone of relaxed constraints $P_r$.
\begin{lmm}
    \label{cont-cone}
Let $P_r = \cone(\U_r)$. Then~\eqref{cond-approx} is equivalent to
\begin{equation}
\label{CNS-cone}
\forall \adjtar \in X, \quad \left(\text{for a.e. } t \in (0,T), \quad L_T^\ast \adjtar(t) \in P_r^\circ\right) \quad \implies  \quad  \langle \tilde y_T, \adjtar \rangle_X \leq 0.
\end{equation}
\end{lmm}

\subsubsection{Exact reachability}
Second, we consider a necessary and sufficient condition for exact reachability, which leads to a constructive control under some additional technical assumptions. 
In terms of constructive approaches, this corresponds to studying the dual functional~\eqref{dual_func} for $\e=0$.
We will be interested in the quantity 
\[c^\star:= \inf \left\{c \geq 0, \;  \forall \adjtar \in X, \; \scl{\tilde y_T}{\adjtar}_X \leq c \, F^\ast(L_T^\ast \adjtar)^{1/2}\right\} = \sup_{\|\adjtar\|_X =1}\frac{\scl{\tilde y_T}{\adjtar}_X}{F^\ast(L_T^\ast \adjtar)^{1/2}}, \]
with the convention $0/0 = 0$, $\alpha/0 = +\infty$ for $\alpha>0$.
Note that $c^\star \in [0,+\infty]$, and $c^\star \in (0,+\infty]$ if and only if $\tilde y_T  \neq 0$. 

In the case of exact reachability, we are led to specify some additional time-regularity controls may/should have, in the form $u \in L^2_{\U_r}$. In this case, we will say that \textit{$\tar$ is exactly reachable from $y_0$ in time $T$
under the constraints $P$ (or $P_r$) with controls in $L^2_{\U_r}$.}

\begin{thrm}
\label{thm-main-exact}
The state $y_f$ is exactly reachable from $y_0$ in time $T>0$ under the cone of relaxed constraints $P_r = \cone(\U_r)$ with controls in $L^2_{\U_r}$ if and only if  
\begin{equation}
\label{cond-exact}\exists c > 0, \; \forall \adjtar \in X, \quad \scl{\tilde y_T}{\adjtar}_X \leq c \, F^\ast(L_T^\ast \adjtar)^{1/2}.
\end{equation}

Now assume that the latter condition holds, \textit{i.e.}, $c^\star < +\infty$. Then  $J_{0}$ admits a minimiser if and only if $c^\star$ is attained, and if so for any such minimiser $\adjtar^\star$, at least one control $u \in \partial F^\ast(L_T^\ast \adjtar^\star)$
steers the solution of~\eqref{control} from $y_0$ to $\tar$ in time~$T$. Furthermore, 
\begin{itemize}
\item there exists a unique such control if and only if~\eqref{extremal} holds,
\item if in addition~\eqref{extremality} holds, the unique control takes values in the original cone $P = \cone(\U)$.
\end{itemize} 
\end{thrm}

Contrarily to the case of approximate reachability, we must consider the additional information that controls are in $L^2_{\U_r}$. 
Let us give one general sufficient condition making this condition superfluous.

If $\U_r$ (which is a closed, bounded and convex set) is such that $0$ is in the interior of $\U_r$ relative to the cone it generates, \ie if
\begin{equation}
\label{regular_generating_set}
\exists \delta>0, \quad  \cone(\U_r) \cap \overline{B}(0,\delta) \subset \U_r,
\end{equation}
then it is easily seen that $j_{\U_r}(u) \leq \delta^{-1}\|u\|_U$ for all $u \in U$. As a result, if $\U_r$ satisfies~\eqref{regular_generating_set}, then $L^2_{\U_r} = E$ so that the additional regularity requirement that controls be in $L^2_{\U_r}$ may safely be removed. We also note that under~\eqref{regular_generating_set}, $P_r = \cone(\U_r)$ is closed, see Appendix~\ref{app-gauge-results}.



\subsubsection{Consequences}
Let us now explain the implications these theorems have, first when the chosen set $\U$ to generate the cone $P$ is convex and closed, second when this assumption is dropped.  

\paragraph{Convex closed case.}
First, assume that $\U$ is convex and closed, in which case relaxation is unnecessary. Hence, $\U_r=\U$ and $P_r = \cone(\U_r) = \cone(\U) = P$, making the assumption~\eqref{extremality} true (since $\U$ is then weakly closed) but pointless. 

In this case, Theorem~\ref{thm-main-approx} shows that whenever $\tar$ is approximately reachable from $y_0$ in time $T>0$ under the constraints $P=P_r$, our method provides a constructive way to obtain controls achieving approximate reachability, provided that~\eqref{extremal} holds.

Theorem~\ref{thm-main-exact} shows that whenever $\tar$ is exactly reachable from $y_0$ in time $T>0$ under the constraints $P=P_r$ with controls in $L^2_{\U_r} = L^2_{\U}$, our method provides a constructive way to obtain controls achieving exact reachability, provided  that $c^\star$ is attained and~\eqref{extremal} holds.

\paragraph{General case.}
When $\U$ is no longer assumed to be convex and closed, Theorem~\ref{thm-main-approx}  and Theorem~\ref{thm-main-exact} yield new reachability results in the following sense.

If $\tar$ is approximately reachable from $y_0$ in time $T>0$ with controls under the cone of relaxed constraints $P_r =\cone(\U_r)$, then $\tar$ is also approximately reachable from $y_0$ to $\tar$ in time $T>0$ under the original constraint cone $P=\cone(\U)$, provided that \eqref{extremal}  and~\eqref{extremality} hold. Furthermore, our method provides a constructive way to build controls achieving approximate reachability.

If $\tar$ is exactly reachable from $y_0$ in time $T>0$ with controls under the cone of relaxed constraints $P_r = \cone(\U_r)$ with controls in $L^2_{\U_r}$, then $\tar$ is also exactly reachable from $y_0$ to $\tar$ in time $T>0$ under the original constraint cone $P=\cone(\U)$, provided that $c^\star$ is attained and that \eqref{extremal} and~\eqref{extremality} hold. Furthermore, our method provides a constructive way to build controls achieving exact reachability.

\paragraph{The bang-bang principle.}
Our results bear a strong connexion to the so-called \textit{bang-bang principle}.
The bang-bang principle is a property that many control systems satisfy, which can be stated as follows
\[R^T_K(y_0)=R^T_{\operatorname{ext}(K)}(y_0),\]
where the notation $R_K^T(y_0)$ stands for the reachable set from $y_0$ in time $T$, under the constraints given by the set $K$.
The above principle hence means that any state that can be reached in time $T$, from a given initial state $y_0$, with controls in a convex compact set $K$, can also be reached with controls in the set of extremal points of $K$. 

It also exists in the weaker form
\[R^T_K(y_0)=\overline{R^T_{\operatorname{ext}(K)}(y_0)}^w,\]
for weakly compact convex sets $K$. We refer to \cite{gugat2008norm} for more details.

An important difference is that our constraint sets are not bounded, as they are cones, but we do consider generating sets which are bounded. Relaxing these constraint sets to their convex hulls allows us to work with a convex optimization framework, in which we recover, under certain general conditions, a form of bang-bang principle.

\paragraph{Condition~\eqref{extremal}.}
Since the central condition~\eqref{extremal} may not be straightforward to check as one has little information about $\adjtar^\star$, one may replace it with the following weaker condition that no longer depends on $\e\geq 0$, $y_0$, $\tar$ and $T>0$:
\begin{equation}
\tag{$\widetilde{\text{H}}$}
\label{extremal_bis}
\forall \adjtar \neq 0,  \quad  B^\ast S_t^\ast \adjtar \notin \sing(\U_r) \cap (P_r^\circ)^c  \text{ for a.e. } t >0.
\end{equation}

Of course,~\eqref{extremal_bis} implies ~\eqref{extremal}. In order to discuss the above conditions, we give the following useful result when it comes to proving that the adjoint trajectory does not spend time within a given set. 

For a vector $b \in X$, we let $\ell(b)$ be the largest integer such that $b \in \mathcal{D}(A^j)$, with the convention $\ell(b) = +\infty$ if $b \in\mathcal{D}(A^j)$ for all $j \in \N$.
\begin{prpstn}
\label{suff-cond-extremal}
    Let $K\subset X$ be any set, and assume that 
    \begin{itemize}
        \item[(i)] the semigroup $(S_t^\ast)$ is injective,
        \item[(ii)] 
    \[\big(\mathrm{Ran}\big\{A^j b, \; b \in K^\perp, \, 0 \leq j \leq \ell(b)\big\}\big)^\perp = \{0\}.\]
    \end{itemize}
    Then the set $\{t \geq 0, \, S_t^\ast \adjtar \in K\}$ is of measure $0$ for any $\adjtar \neq 0$. 
\end{prpstn}
The proof of Proposition~\ref{suff-cond-extremal} is given in Appendix~\ref{proof-avoid-set}.

Recall that $B^\ast q \in \sing(\U_r)$ if and only if $q \in \sing(B\U_r)$, and note that $B^\ast q \in (P_r^\circ)^c$ if and only if $q \in ((B P_r)^\circ)^c$.
Hence, one can in practice try and apply Proposition~\ref{suff-cond-extremal} to $K = \sing(B \, \U_r) \cap ((B P_r)^\circ)^c$ to obtain~\eqref{extremal_bis}. However, it is important to note that Proposition~\ref{suff-cond-extremal} only sees the closure of the subspace spanned by $K$, in the sense that $K$ satisfies the required hypotheses if and only if $\overline{\mathrm{Ran}(K)}$ does, since $\overline{\mathrm{Ran}(K)}$ and $K$ have the same orthogonal complement. As a result, it will usually be more appropriate to use Proposition~\ref{suff-cond-extremal} to subcones that constitute the cone $\sing(B \U_r)\cap ((B P_r)^\circ)^c$. 

In practical cases and in finite dimension, there will indeed typically be finitely many such subcones; see Subsection~\ref{subsec-toy-example} for an example.

\subsection{Extensions and perspectives}
\label{subsec-ext-persp}

\paragraph{Alternative cost.}
It would be possible to consider a slightly different cost than~\eqref{primal-cost}, namely 
\begin{equation}
\label{primal-cost-linfty}
F(u) := \tfrac{1}{2} \sup_{t \in (0,T)} j_{\U_r}^2(u(t)).
\end{equation}
which would lead to the dual functional $J_\e$ as before but with $F^\ast$ given by
\begin{equation}
    \label{dual-cost-linfty}
F^\ast(p) = \tfrac{1}{2} \left(\int_0^T \sigma_{\U_r}(p(t))\,dt\right)^2.
\end{equation}

One advantage is that one would obtain controls with a fixed amplitude rather than a time-varying one, which can be an interesting feature for applications. In fact, this is done in~\cite{Nous} for a specific control problem.

However, this alternative cost would lead to several adjustments; in particular, the more natural functional setting for generalising all our results to~\eqref{dual-cost-linfty} would be $L^\infty(0,T;U)$, which has a less natural dual structure than $E = L^2(0,T;U)$. In order to lighten our presentation, we have chosen not to do so. 

\paragraph{Open questions.}
In the convex case, our variational method allowed us to derive sufficient conditions for approximate or exact reachability in cone. We have further proved that these conditions happen to be necessary, by other means.

In the non-convex case, relaxation yields a set of sufficient conditions for approximate or exact reachability in nonconvex cones. 
This leaves open the question of the necessity of these conditions. 

Accordingly, studying the necessity of these conditions will allow for a more detailed picture of the bang-bang principle in infinite dimensions, with undounded constraints. Indeed, finding counterexamples, or proving that these conditions are necessary, will provide an in-depth understanding of failures of the bang-bang principle in infinite dimension.

\bigskip 

\paragraph{Outline of the paper.}
Our article is split into two parts. The purpose of Section~\ref{sec-examples} is to apply our method to various examples, which leads to the generalisation of previous works as well as to new results.
Then, Section~\ref{sec-proofs} compiles the proofs of all our results, building upon the Fenchel-Rockafellar theorem.  
\section{Examples and applications}
\label{sec-examples}
We discuss the application of our method to four examples.
\begin{itemize}
\item we show that the HUM method is a particular case of our methodology.
\item we analyse a toy example in small dimension with non-convex constraints. We explain in full detail how to properly follow the different steps underlying our method. 
\item we study abstract general finite-dimensional control problems under the $k$-sparsity constraints. This generalises an approach developed in~\cite{zuazua2010switching} for $k=1$.
\item we discuss approximate reachability for control problems in $L^2(\Omega)$. First, we consider nonnegativity constraints and then proceed to adding specific sparsity constraints.
We recover the result of~\cite{Nous} regarding the on-off shape control of parabolic equations, with the subtle difference that controllers have time-varying amplitude. 
\end{itemize}

\subsection{Unconstrained case}
\label{subsec-unconstrained-example}

\subsubsection{HUM method}
We here assume that $\U$ is the unit ball $\{u \in U, \; \|u\|_U \leq 1\}$, which obviously falls in the convex case with $\U_r = \U$. 

Then the cone $\cone(\U_r) = P_r = P = \cone(\U)$ is the whole of $U$: we are in the unconstrained case. In this setting, there are numerous sufficient conditions for approximate (resp. exact) controllability to hold, yielding approximate (resp. exact) reachability independently of $y_0$, $\tar$ and, possibly, $T>0$.
In this case, we find $\sigma_{\U}(u) = j_\U(u)= \|u\|_U$, meaning that $F^\ast(p) = \tfrac{1}{2} \|p\|_E^2$, hence the functional $J_{\e}$ boils down to 
\begin{align}    J_{\e}(\adjtar) = \tfrac{1}{2} \int_0^T \|L_T^\ast \adjtar(t)\|_U^2\, dt - \langle \tilde y_T, \adjtar\rangle_X + \e \|\adjtar\|_X 
=\tfrac{1}{2} \|L_T^\ast \adjtar\|_E^2 - \langle \tilde y_T, \adjtar\rangle_X + \e \|\adjtar\|_X
\end{align}
and the corresponding optimisation problem is given by
\[\inf_{u \in E, \;  y(T) \in \overline{B}(\tilde y_T, \e)} \tfrac{1}{2}\int_0^T \|u(t)\|_X^2 \,dt = \inf_{u \in E, \;  y(T) \in \overline{B}(\tilde y_T, \e)} \tfrac{1}{2}\|u\|_{E}^2.\]
For $\e=0$, we recover the functional underlying the so-called Hilbert Uniqueness Method, introduced by Lions in~\cite{lions-1988}. For $\e>0$, we recover the functional introduced  by Lions in~\cite{Lions-1992} to study approximate controllability.

Whenever a dual optimal variable exists, it gives rise to a unique control through~\eqref{opt_control}, since the latter equation then amounts to $u^\star = L_T^\ast \adjtar^\star$.

\subsubsection{Exact controllability}

As is well known, exact controllability at time $T$ (\ie exact reachability for any $y_0, \tar \in X$ in time $T$) is equivalent to the observability inequality: 
\begin{equation}
    \label{obs}
\exists C >0, \, \forall \adjtar \in X, \quad \|\adjtar\|_X^2 \leq C^2 \|L_T^\ast \adjtar\|_E^2
\end{equation}
Furthermore, it is also well known that one may then achieve exact controllability by minimising the functional $J_{0}$, that is, the dual functional attains its minimum. Indeed, it is easily seen to be coercive in this case. 

Let us explain how our framework recovers this case as well: we have $c^\star< +\infty$ and $c^\star$ is attained, whatever the choice of $y_0, \tar$ in $X$.


\begin{prpstn} 
\label{obs-c2}
Let $T>0$ be fixed, and assume that~\eqref{obs} holds. Then whatever $y_0$, $\tar$ in $X$, $c^\star<+\infty$ and it is attained. 
\end{prpstn}
\begin{proof}
We let $y_0$, $\tar$ be fixed in $X$.
The first statement is readily obtained by applying the Cauchy-Schwarz inequality:
\[\forall \adjtar \in X, \quad \langle \tilde y_T, \adjtar \rangle_X \leq \|\tilde y_T\|_X \|\adjtar\|_X \leq C \|\tilde y_T\|_X \|L_T^\ast \adjtar\|_E =\sqrt{2} C  \|\tilde y_T\|_X F^\ast(L_T^\ast \adjtar)^{1/2}, \]
this shows that $c^\star \leq\sqrt{2} C \|\tilde y_T\|_X < +\infty$.

Now let us prove that $c^\star$ is attained. The case $\tilde y_T = 0$ (which corresponds to $S_T y_0 = \tar$ or, equivalently, to $c^\star = 0$) is obvious. 

In the interesting case where $\tilde y_T \neq 0$ or equivalently $c^\star >0$, let $(\adjtar^n)$, $\|\adjtar^n\|_X = 1$ be a maximising sequence, \textit{i.e.}, $\tfrac{\langle \tilde y_T, \adjtar^n \rangle_X}{\|L_T^\ast \adjtar^n\|_E}$ converges to $c^\star$.
Upon extraction, we may assume that $\adjtar^n \rightharpoonup \adjtar$ in $X$. The numerator hence converges, and we must have $\langle \tilde y_T, \adjtar \rangle_X \geq 0$ for the quotient to converge to a positive value. Now in view of~\eqref{obs}, we have $\|L_T^\ast \adjtar^n\|_E \geq C^{-1} \|\adjtar^n\|_X = C^{-1}$, hence the denominator is bounded away from $0$. This proves that the numerator actually has to converge to a positive value, namely $\langle \tilde y_T, \adjtar \rangle_X > 0$. In particular, we have $\adjtar \neq 0$ and hence, again by~\eqref{obs}, $L_T^\ast \adjtar \neq 0$.

By weak lower semicontinuity of the norm and given that $L_T^\ast \adjtar^n \rightharpoonup L_T^\ast \adjtar$ in $E$, we find
\[ \frac{\langle \tilde y_T, \adjtar \rangle_X}{\|L_T^\ast \adjtar\|_E} \geq \limsup \frac{\langle \tilde y_T, \adjtar^n \rangle_X}{\|L_T^\ast \adjtar^n\|_E} = c^\star,\]
showing that $\adjtar$ (or more precisely, $\tfrac{\adjtar}{\|\adjtar\|_X}$) reaches the supremum $c^\star$.

\end{proof}

\subsection{A finite-dimensional example}
\label{subsec-toy-example}
\subsubsection{Setting}
We are in the case where $X = \R^3$, $U = \R^2$, with 
\begin{equation}
\label{toy-matrices} 
A = \begin{pmatrix}
1 & 2 & 0\\
1 & -1 & 2 \\
1 & 1 & -1
\end{pmatrix},  \qquad B = \begin{pmatrix}
1 & 0 \\
-1 & 1 \\
0 & 0 \\
\end{pmatrix}
\end{equation}
It is easily seen that the pair $(A,B)$ satifies Kalman's rank condition. Hence for any $y_0, \tar \in X$ and $T>0$, $\tar$ is exactly reachable from $y_0$ in time $T$ (in the absence of constraints).

Now consider the following cone as a constraint set 
\begin{equation}
\label{toy-constraint}
P:= \{(u_1, u_2) \in \R^2, \;  u_1 \geq 0 \text{ or } u_1 \leq 0, \; u_2 = 0\}.
\end{equation}

We generate the cone $P$ by intersecting with the unit ball of $\R^2$, that is, we set
\[\U:= P \cap \overline{B}(0,1).\]
The resulting set is not convex, hence we form the relaxed constraint set $\U_r$, which is given by 
\[\U_r = \left(\overline{B}(0,1) \cap \{(u_1, u_2) \in \R^2, \, u_1 \geq 0\}\right) \cup \left(\overline{B}_1(0,1) \cap \{(u_1, u_2) \in \R^2, \, u_1 \leq 0\}\right),\]
where we use the shorthand notation $\overline{B}_p(0,1)$ for the unit ball associated to the $\ell^p$ norm.

Note that the cone $P_r$ generated by $\U_r$ is the whole of $\R^2$. Hence, for any $y_0, \tar \in X$ and $T>0$, $\tar$ is exactly reachable from $y_0$ in time $T$ under the constraints $P_r$, since this amounts to not having constraints at all. These sets are illustrated by Figure~\ref{fig:toy_example}.

\begin{figure}[h]
     \centering
     \begin{subfigure}[b]{0.42\textwidth}
         \centering
         \includegraphics[width=\textwidth]{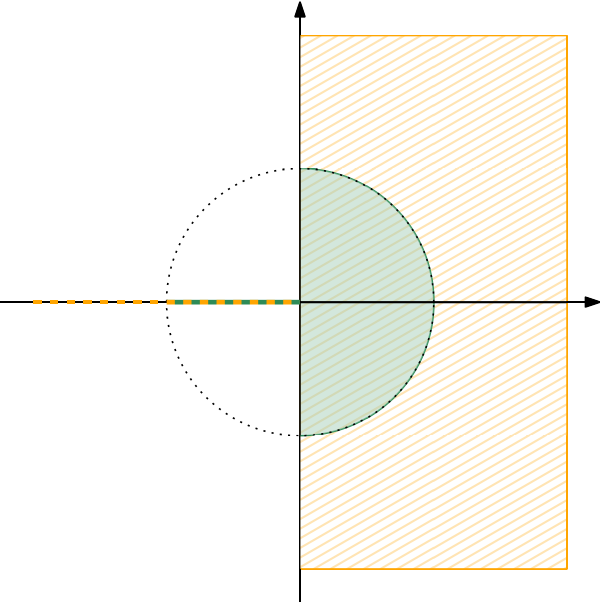}
     \end{subfigure}
     \hfill
     \begin{subfigure}[b]{0.42\textwidth}
         \centering
         \includegraphics[width=\textwidth]{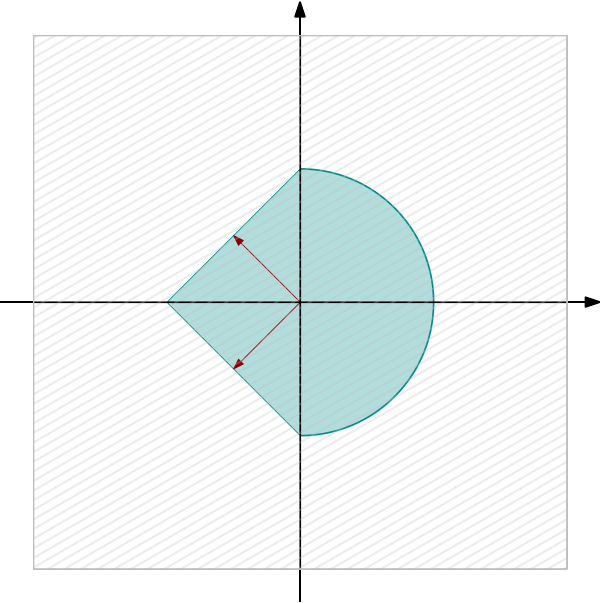}
     \end{subfigure}
     \hfill
        \caption{The left figure shows the constraint cone $P$ (hatched pale orange area) with the set $\U$ chosen to generate it (in green). The right figure shows the relaxed constraint set  $\U_r$ (in cyan). The cone $P_r$ it generates is the whole of $\R^2$ (the hatched area in gray). The two vectors (in magenta) generate the cone of singular vectors $\sing(\U_r)$.}
        \label{fig:three graphs}
    \label{fig:toy_example}
\end{figure}

\subsubsection{Functional and optimality conditions}
We focus on on exact reachability, hence we set $\e=0$.
First, we compute the gauge and support functions of $\U_r$:
\[\forall u \in \R^2, \qquad j_{\U_r}(u) = 
\begin{cases} 
\|u\|_1 &\text{ if } u_1 \leq 0 \\
\|u\| &\text{ if } u_1 > 0\\
\end{cases}
,\qquad 
\sigma_{\U_r}(u) = 
\begin{cases} 
\|u\|_\infty &\text{ if } u_1 \leq 0 \\
\|u\| &\text{ if } u_1 > 0\\
\end{cases}.
\]

These computations being made, they completely define the primal and dual optimisation problems. Now let us make the optimality condition~\eqref{opt_control} more explicit. It easily seen that $\sing(\U_r)$ is given by two half-lines, generated by $f_1 := (-1,1)$ and $f_2:= (-1,-1)$.
If $q \notin \sing(\U_r)$, the inclusion  $u \in  \sigma_{\U_r}(q) \, \partial \sigma_{\U_r}(q)$ rewrites as follows:  
\begin{equation}
\label{optimality-toy} 
\text{if } q_1 >0, \; u = q, \qquad \text{and} \qquad \text{if } q_1\leq 0, \; u = \|q\|_\infty \; \begin{cases} (0,1) &\text{ if }\, q_2 >  -q_1 \\
(-1,0) &\text{ if } \, q_1 < q_2 < -q_1\\
(0,-1) & \text{ if } \, q_2 < q_1
\end{cases}.
\end{equation}

\subsubsection{Exact reachability}
Now we analyse exact reachability under the constraints $P$, and we do so by applying Theorem~\ref{thm-main-exact} to prove exact reachability (and the fact that controls achieving the target may be obtained by minimising the corresponding defined functional). First note that the relaxed constraint set $\U_r$ clearly satisfies~\eqref{regular_generating_set}, so any reference to the set $L^2_{\U_r}$ is unnecessary when dealing with exact reachability, i.e. $L^2_{\U_r} = E$.

Since Kalman's condition is satisfied, the observability inequality~\eqref{obs} holds and by Proposition~\ref{obs-c2}, we infer that~\eqref{cond-exact} holds as well and $c^\star$ is attained. Furthermore, since $\U$ is closed, we know that~\eqref{extremality} holds.

All that is left to do is to prove that~\eqref{extremal} holds. In order to do so, we use Proposition~\ref{suff-cond-extremal}, which we should apply to (well-chosen subsets of) the set $\sing(B \U_r) \cap ((B P_r)^{\circ})^c$. In fact, it will be enough to consider $\sing(B \U_r)$, which we now compute.

Recalling that $p \in \R^3$ is in $\sing(B \U_r)$ if and only if $B^\ast p\in \sing(\U_r)$, we find that
\[\sing(B\U_r) = \{p \in \R^3, \; p_1 = 0, p_2\geq 0\} \cup  \{p \in \R^3, \; p_1=2p_2, \, p_1\leq 0\}.\]

We apply Proposition~\ref{suff-cond-extremal}: at this stage, it is important not to use it directly at the level of $K = \sing(B \U_r)$, simply because $\mathrm{Ran}(K) = \R^3$, hence $K^\perp = \{0\}$ and there will be no non-zero vector $b \in K^\perp$. Instead, we denote \[K_1 :=\{p \in \R^3, \,  p_1 = 0, \, p_2\geq0\}, \quad K_2 := \{p \in \R^3, \, p_1=2p_2, \, p_1\leq 0\}\]
so that $\sing(B\U_r) = K_1 \cup K_2$, and we apply~Proposition~\ref{suff-cond-extremal} to the two sets $K_1$ and $K_2$ separately. 

For $b_1 = (1,0,0) \in K_1^\perp$, we find that $(b_1, Ab_1, A^2 b_1)$ is a family of rank $3$. Similarly for $b_2 = (1,-2,0) \in K_2^\perp$, we find that $(b_2, Ab_2, A^2 b_2)$ is a family of rank $3$.

Consequently, applying Theorem \ref{thm-main-exact}, we have proved:
\begin{prpstn} 
Consider the linear control system defined by the matrices $(A,B)$ given in \eqref{toy-matrices}. Then, whatever $y_0$, $\tar$ in $\R^3$, $T>0$ are, $\tar$ is exactly reachable from $y_0$ in time $T$ under the constraints $P$ given by~\eqref{toy-constraint}.    

Furthermore, let $\adjtar^\star$ be any minimiser of $J_0$ on $\R^n$. Then, letting $q(t) = L_T^\ast \adjtar^\star$, we have $q(t) \notin \sing(\U_r)$ for a.e. $t \in (0,T)$. The unique control defined by the formula $u(t) \in \sigma_{\U_r}(q(t))\,  \partial \sigma_{\U_r}(q(t))$ according to~\eqref{optimality-toy} steers $y_0$ to $\tar$ in time $T>0$, and takes values in $P$.
\end{prpstn}

\subsection{Sparse controls in finite dimension}
We are in the finite dimensional case $X = \R^n$, $U =\R^m$. Here, we apply our general methodology to the case where controls must be $k$-sparse for some $1 \leq k \leq m-1$, \ie at most $k$ components of $u$ must be active at almost all times $t \in (0,T)$.  We focus on exact reachability.

As will be seen, the work~\cite{zuazua2010switching} about so-called \textit{switching controllers} is a particular case of this general framework for $k=1$.

\subsubsection{Setting}
Given $y_0, \tar, T$, the goal is to steer $y_0$ to $\tar$ with controls that are $k$-sparse, namely 
\begin{equation}
\label{k-sparse}
\text{ for a.e. } t \in (0,T), \quad
\|u(t)\|_0 \leq k
\end{equation}
where the $\ell^0$ (semi-)norm $\|\cdot\|_0$ refers to the number of non-zero components of a given vector. 

We correspondingly define the constraint set
\[P^{(k)}:= \{u \in \R^m, \; \|u\|_0 \leq k\}, \]
which is obviously a closed cone. This cone is non-convex since $k \leq m-1$.  

We may now generate the cone $P^{(k)}$ as follows: 
\[\U^{(k)}:= P^{(k)} \cap \overline{B}_\infty(0,1) = \{u \in \R^m, \; \|u\|_0 \leq k, \, \|u\|_\infty \leq 1\}.\]
$\U^{(k)}$ is not convex but it is closed, hence~\eqref{extremality} holds.

\begin{rmrk}
One could also generate $P^{(k)}$ by intersecting it with the unit ball associated to the Euclidean norm $\|\cdot\|_2$. This would make the convex hull $\U^{(k)}_r$ and the set of its singular normal vectors $\sing(\U^{(k)}_r)$ less tractable, while not making the latter set substantially smaller~\cite{argyriou2012sparse}.
\end{rmrk}

\subsubsection{Functional}
For an account of some of the basic results used here, we refer to~\cite{hiriart-poly-norms}.
Following our general method, we first compute the relaxed constraint set associated to $\U^{(k)}$, which is given by
\[\U^{(k)}_r = \{u \in \R^m,\; \|u\|_\infty \leq  1, \, \|u\|_1 \leq k \},\]
which naturally appears as the unit ball for a norm. In particular, $j_{\U^{(k)}_r}$ is defined by 
\[\forall u \in \R^m, \quad j_{\U^{(k)}_r}(u)= \max\Big(\frac{\|u\|_1}{k}, \|u\|_\infty\Big).\]  
We also note that the cone $P_r^{(k)}$ generated by $\U^{(k)}_r$ is the whole of $\R^m$, whatever the value of $1 \leq k \leq m-1$, and that $\U_r^{(k)}$ satisfies~\eqref{regular_generating_set}, meaning that we may drop the reference to any $L^2_{\U^{(k)}_r}$ regularity altogether since this set equals $E$.

We are now ready to compute the corresponding primal and dual functionals; since we focus on exact reachability, we set $\e = 0$.

For a vector $u \in \R^m$ we shall denote $|u_{(1)}|\geq   |u_{(2)}|\geq  \ldots \geq |u_{(m)}|$ by reordering the components of $|u|$ in a decreasing way. Then one has (see \cite{hiriart-poly-norms})
\[\sigma_{\U^{(k)}_r}(u) = \sum_{i=1}^k |u_{(i)}|,\]
that is, $\sigma_{\U^{(k)}_r}(u)$ is the $\ell^1$ norm of the vector $(u_{(1)}, \ldots u_{(k)}) \in \R^k$. As expected, $\sigma_{\U^{(k)}_r}(u)$ boils down to the $\ell^\infty$ norm for $k=1$ (and to the $\ell^1$ norm for $k=m$).


All in all, the cost associated to the optimal problem is given by 
\[\forall u \in E, \quad F(u) =  \tfrac{1}{2}\int_0^T j^2_{\U^{(k)}_r}(u(t)) \,dt.\]
For $p_f\in \R^n$, generically denoting $p(t) = S_{T-t}^\ast \adjtar$, the functional underlying the dual problem is 
\begin{align}
\begin{split}
\label{dual-sparse}
    J_{0}(\adjtar) = 
\tfrac{1}{2}\int_0^T \sigma_{\U^{(k)}_r}^2(L_T^\ast \adjtar(t)) \,dt - \langle \tilde y_T, \adjtar\rangle_{\R^n} 
=\tfrac{1}{2} \int_0^T \Big(\sum_{i=1}^k 
|(B^\ast p(t))_{(i)}|
\Big)^2 dt -\langle \tilde y_T, \adjtar\rangle_{\R^n}.
\end{split}
\end{align}
The above dual functional is exactly  the one introduced in~\cite{zuazua2010switching} in the case $k=1$. 
\subsubsection{Results}
Now let us apply our Theorem~\ref{thm-main-exact}, and compare to the main result Theorem 2.3 of~\cite{zuazua2010switching} for $k=1$. 

Recall that the cone of relaxed constraints is $P_r^{(k)} = \R^m$, and we already know that~\eqref{extremality} holds. As a result, if
\begin{itemize}
    \item $\tar$ is exactly reachable from $y_0$ to $\tar$ in time $T>0$,
    \item $c^\star$ is attained,
    \item assumption \eqref{extremal} holds,
\end{itemize}
then $\tar$ is exactly reachable from $y_0$ in time $T$ under the switching constraint~$P^{(k)}$ given by~\eqref{k-sparse}.

Now let us discuss a sufficient condition for the above three hypotheses to hold. Of course, the first one is true in particular if controllability (without constraints) is satisfied, \textit{i.e.}, if the pair $(A,B)$ satisfies Kalman's rank condition. 
In view of Proposition~\ref{obs-c2}, $c^\ast$ is attained under that same assumption.

In order to find a sufficient condition for the third condition~\eqref{extremal} to hold, we examine~\eqref{extremal_bis}. 
\begin{lmm}
    For all $k \leq m-1$, there holds
\begin{align*} \sing(\U_r^{(k)})
 = \{u \in \R^m,\;\, |u_{(k)}| = |u_{(k+1)}|\}.
\end{align*}
Furthermore, if $u \notin \sing(\U_r^{(k)})$, the unique $v \in \argmax_{v \in \U^{(k)}_r} \langle u, v \rangle$ is given by \[v = \mathrm{sign}(u) \mathbf{1}_{J(u)}, \qquad J(u) := \{j \in \{1, \ldots, m\}, \,|u_j| > |u_{(k+1)}|\}.\]
\[v_k=\begin{cases}
    \mathrm{sign}(u_k) \ &\textrm{if} \ k\in J(u),\\
    0 \ &\textrm{otherwise}
\end{cases}\]
\end{lmm} 
\begin{proof}
Let us temporarily denote $S_k =\{u \in \R^m,\;\, |u_{(k)}| = |u_{(k+1)}|\}$. Let $u \notin S$ be fixed; we shall prove that $u \notin \sing(\U_r^{(k)})$.  By assumption, there holds $|u_{(k)}| > |u_{(k+1)}|$. Let $v$ be a corresponding maximiser, that is $v$ such that 
\[\max_{v \in \U_k} \langle u, v \rangle = \sigma_{\U^{(k)}_r}(u) = \sum_{i=1}^k |u_{(i)}|.\] 
One easily sees that $v$ is uniquely determined (hence $u \notin \sing(\U_r^{(k)})$) and is given by $\mathrm{sign}(u) \mathbf{1}_{J(u)}$. 

Conversely, let $u \in S_k$, and let us prove that $u \in \sing(\U_r^{(k)})$. 
Let $i_1, \ldots, i_{k+1}$ be $k+1$ distinct indices in $\{1, \ldots,m\}$ such that $|u_{i_j}| = |u_{(j)}|$ for $j \leq k+1$. Then define $v$ and $w$ be such that $v_{i_j} = w_{i_j} = \mathrm{sgn}(u_{i_j})$ for $1 \leq j \leq k-1$, and $v_{i_k} = \mathrm{sgn}(u_{i_k})$, $w_{i_{k+1}} = \mathrm{sgn}(u_{i_{k+1}})$ if $u_{(k)} \neq 0$, or $v_{i_k} = 1$, $w_{i_{k+1}} = 1$ if $u_{(k)}= 0$. All the unmentioned components of $v$ and $w$ are taken to be equal to $0$. Clearly, we have $v, w \in \U_r^{(k)}$, and $v \neq w$ in both cases. These two different vectors satisfy 
\[\langle u, v \rangle =\langle u, w \rangle = \sigma_{\U^{(k)}_r}(u) = \sum_{i=1}^k |u_{(i)}|,\]
proving that $u \in \sing(\U_r^{(k)})$.
\end{proof}

Thanks to the Lemma, we may assert that assumption~\eqref{extremal_bis} is more explicitly given by 
\begin{equation}
    \label{H-sparse}
\forall \adjtar \neq 0, \quad \{t \in (0,T), \; |(B^\ast p(t))_{(k)}| = |(B^\ast p(t))_{(k+1)}|\},\; \text{ has measure $0$},
\end{equation} where $p(t) = S_{T-t}^\ast \adjtar$.

We arrive at the following result:
\begin{prpstn}
\label{cond_reachable_sparse}
Assume that $(A,B)$ satisfies Kalman's rank condition, and further that~\eqref{H-sparse} holds. Then for any $y_0$, $\tar$ in $\R^n$ and $T>0$, $\tar$ is exactly reachable from $y_0$ in time $T$ under the constraints~$P^{(k)}$ given by~\eqref{k-sparse}.

Furthermore, denoting $\adjtar^\star$ a minimiser of~\eqref{dual-sparse} and $p^\star(t) = S_{T-t}^\ast \adjtar^\star$,  the control defined for $t \in (0,T)$ by
\[u^\star(t) :=  \bigg(\sum_{i \in J(B^\ast p^\star(t))} |(B^\ast p^\star(t))_i|\bigg)  \, \mathrm{sign}(B^\ast p^\star(t))  \, \mathbf{1}_{J(B^\ast p^\star(t))}\]
is $k$-sparse and drives $y_0$ to $\tar$ in time $T$.
\end{prpstn}

Let us discuss the assumption underlying Proposition~\ref{cond_reachable_sparse}.
It will be convenient to write $B = (b_1, \ldots, b_m)$ with $b_j \in \R^n$, so that $(B^\ast p)_i = \langle b_i, p \rangle$ for $i \in \{1, \ldots, m\}$ and $p \in \R^n$.

With this notation in place, the assumption of interest
is clearly satisfied if the $\tfrac{1}{2}m(m-1)$ sets (indexed by $1 \leq j,\ell \leq m$ with $j \neq \ell$) $\{t \in (0,T), \; |\langle b_i, p(t) \rangle| = |\langle b_j, p(t) \rangle| \}$ are all of measure $0$ whatever $\adjtar \neq 0$ and $T>0$ are. If so, the assumption is verified for any $1 \leq k \leq m-1$. As remarked in~\cite{zuazua2010switching}, this unique continuation property is satisfied as soon as 
\begin{equation}
\label{strong_kalman}
\forall \; 1 \leq j \neq \ell \leq m, \quad (A, b_j \pm b_\ell) \text{ satifies Kalman's rank condition},
\end{equation}
which in fact is (much) stronger than the assumption that the pair $(A,B)$ satisfies Kalman's rank condition.  

We arrive at the following corollary giving a sufficient strong condition for exact reachability under the constraints $P^{(1)}$ (and hence $P^{(k)}$ for all $1 \leq k \leq m$).
\begin{crllr}
Assume that the condition~\eqref{strong_kalman} holds. Then for any $y_0$, $\tar$ in $\R^n$ and $T>0$, $\tar$ is exactly reachable from $y_0$ in time $T$ under the constraints $P^{(1)}$. 
\end{crllr}

\subsection{Nonnegative and on-off shape control}
\subsubsection{Objective}
Here, we show how our methodology allows to recover the results of~\cite{Nous} for the approximate reachability of parabolic equations by shape controls, except that the latter work is concerned with a variation of the present technique, see subsection~\ref{subsec-ext-persp}.

Given $\Omega$ a smooth domain, we consider the control problem~\eqref{control} with $B = \mathrm{Id}$, $X = U = L^2(\Omega)$. Given $0<m_L \leq |\Omega|$, we are interested in the approximate reachability problem by means of so-called \textit{on-off shape controls}, \textit{i.e.}, controls that write for a.e. $t \in (0,T)$ as $u(t) = M(t) \chi_{\omega(t)}$ for some $M(t)>0$ and $\omega(t)$ a measurable set of measure $|\omega(t)| \leq m_L$.

Hence, we naturally set 
\begin{equation}
\label{shape-constraint}
\U:= \{\chi_{\omega}, \; |\omega| \leq m_L\},
\end{equation}
whose associated generated cone $P = \cone(\U)$ is the wanted constraint set. 

There is a direct analogy with the finite-dimensional case above with $k$-sparse controls, with the following additional difficulties 
\begin{itemize}
    \item we are in infinite dimension,
    \item there is a nonnegativity constraint on top of the sparsity one,
    \item controls must be constant on their support.
\end{itemize}
\subsubsection{Primal optimisation problem}

In this case, we compute 
\[\U_r = \left\{u \in L^2(\Omega), \; 0 \leq u\leq 1 \text{ and } \int_\Omega u \leq m_L\right\}.\]
We note that the corresponding generated cone $P_r = \cone(\U_r)$ is the set of (essentially bounded) non-negative controls, \textit{i.e.}, $P_r= \{u \in L^\infty(\Omega), \; u \geq 0\}$. Even though $\U_r$ is closed, $P_r$ is not.
What is more, it is easily seen that in this case, $\ext(\U_r) = \U$, hence~\eqref{extremality} holds.\footnote{Proof. If $u = \chi_{\omega}$ with $|\omega| \leq m_L$, then any decomposition $\chi_\omega = t v_1 + (1-t) v_2$ with $v_1, v_2 \in \U_r$, $0<t<1$ leads to $v_1 = v_2 = 1$ on $\omega$ and $v_1 = v_2 = 0$ on $\Omega \setminus \omega$, \textit{i.e.}, $v_1 = v_2 = \chi_\omega$, which shows that $\chi_\omega$ is extremal in $\U_r$. Conversely, assume by contradiction the existence of $u$ extremal in $\U_r$ but not of the form $\chi_{\omega}$ with $|\omega| \leq m_L$. Then we may find for $\delta>0$ small enough a measurable set $\omega_0$ with $0<|\omega_0| < |\Omega|$ such that $\delta \leq u \leq 1-\delta$ on $\omega_0$. We define $v_1$ and $v_2$  by $v_1 = u + \tfrac{\e}{|\Omega \setminus \omega_0|}$, $v_2 = u - \tfrac{\e}{|\Omega \setminus \omega_0|}$ on $\Omega \setminus \omega_0$, $v_1 = u-\tfrac{\e}{|\omega_0|}$, $v_2 = u+\tfrac{\e}{|\omega_0|}$ on $\omega_0$. These functions do satisfy $v_1, v_2 \in \U_r$ provided that $\e$ be chosen small enough. Furthermore, we have $u = \tfrac{1}{2}(v_1 + v_2)$, a contradiction with the extremality of $u$.}

Finally, computing the gauge of $\U_r$ leads to
\[ \forall u \in L^2(\Omega), \quad j_{\U_r}(u) =    \max \Big(\|u\|_{L^\infty(\Omega)}, \tfrac{\|u\|_{L^1(\Omega)}}{m_L}\Big) + \delta_{\{u \geq 0\}}.\]
The support function of $\U_r$ and its subdifferential, needed to define the dual functional $J_\e$ and to derive optimal controls through the optimality conditions~\eqref{opt_control}, are given in~\cite{Nous}.
The corresponding cost is given for $u \in E$ by
\[F(u) = \tfrac{1}{2} \int_0^T \left(\max \Big(\|u(t)\|_{L^\infty(\Omega)}, \tfrac{\|u(t)\|_{L^1(\Omega)}}{m_L}\Big)^2 + \delta_{\{u(t) \geq 0\}}\right)\,dt.\]

\subsubsection{Approximate reachability in the relaxed set}
Now let us apply Theorem~\ref{thm-main-approx}, fixing $y_0, \tar \in L^2(\Omega)$ and $T>0$ such $\tar \geq S_T y_0$, \ie $\tilde y_T \geq 0$.

First, we check that $\tar$ is approximately reachable from $y_0$ in time $T$ under the constraint $P_r = \cone(\U_r) = \{u \in L^\infty(\Omega), \, u \geq 0\}$. To do so, we analyse condition~\eqref{cond-approx}. 

We let $\adjtar$ be such that $F(L_T^\ast \adjtar) = 0$, then we have $\sigma_{\U_r}(L_T^\ast \adjtar(t)) = 0$ for a.e. $t \in (0,T)$. It is easily seen that $\sigma_{\U_r}(z) =0 \implies z \leq 0$ on $\Omega$, without having to derive an expression for $\sigma_{\U_r}$. Hence we find $L_T^\ast \adjtar = e^{(T-t) A^\star} \adjtar \leq 0$ for a.e. $t \in (0,T)$. By continuity of the (adjoint) semigroup trajectory, putting $t=T$ in the above leads to $\adjtar \leq 0$. Hence, we indeed obtain
\[\scl{\tilde y_T}{\adjtar}_{L^2(\Omega)} =  \scl{\tar -S_T y_0}{\adjtar}_{L^2(\Omega)} \leq 0.\]

At this stage, note that no assumption whatsoever has been made about the operator $A$.

\subsubsection{Approximate reachability in the original set}
We continue applying Theorem~\ref{thm-main-approx} to obtain (approximate) reachability results in $P=\cone(\U)$, arriving at the following general result. 
\begin{prpstn}[\cite{Nous}]
Assume that $\tar \geq S_T y_0$. 
If the adjoint semigroup satisfies the property 
\[\forall \adjtar \neq 0, \; \quad \text{\quad the level sets of $S_t^\ast \adjtar$ have measure $0$ for a.e. $t>0$},\]
then $\tar$ is approximately reachable from $y_0$ in time $T$ under the constraint $P = \cone(\U)$ with $\U$ given by~\eqref{shape-constraint}.
\end{prpstn}
\begin{proof}
Since we have already shown approximate reachability in $\cone(\U_r)$ in the case $\tar \geq S_T y_0$, and because~\eqref{extremality} holds, we only need to prove~\eqref{extremal}. We do so by establishing~\eqref{extremal_bis}, that is we prove that for $\adjtar \neq 0$, $B^\ast S_t^\ast \adjtar = S_t^\ast \adjtar$ avoids $\sing(\U_r)$ for a.e. $t>0$. 
This is a direct application of the bathtub principle~\cite{Nous}[Lemma 2.3], according to which any $v \in \sing(\U_r)$ must have at least one level set of positive measure. 
\end{proof}

\section{Proofs of the main results}
\label{sec-proofs}

\subsection{Organisation}
This section is devoted to the proof of our main results Theorem~\ref{thm-main-approx} and Theorem~~\ref{thm-main-exact}. Throughout, $y_0, \, \tar \in X$ and $T>0$ are fixed, and $\tilde y_T = \tar - S_T y_0$. It is organised as follows: first, we prove by means of Proposition~\ref{reach-implies-cond} in subsection~\ref{necessary-reach} the necessity of the two reachability conditions~\eqref{cond-approx} and~\eqref{cond-exact} for the corresponding reachability statements. 

Then in the next subsections, we analyse the primal optimisation problem~\eqref{ocp}, given by
\[\inf_{u \in E, \; y(T) \in \overline{B}(\tar, \e)} F(u) = \inf_{u \in E, \; L_T u \in \overline{B}(\tilde y_T, \e)} F(u).\]
First, we note that if $\|\tilde y_T\|_X \leq \e$ (which is equivalent to $\tar \in \overline{B}(S_T y_0,\e)$), then the null-control $0 \in E$ is optimal and the infimum above is $0$. In fact, it is then the only optimal control, because
if $F(u) = 0$, we must have $j_{\U_r}(u(t)) = 0$ and hence $u(t) = 0$ for a.e. $t \in (0,T)$, by Lemma~\ref{basic-gauge}[(i)] applied to~$\U_r$. Obviously, the converse also holds: if $u=0$ is optimal (in which case it is the only optimal control), then $\|y_T\|_X \leq \e$.

As explained in the introduction, we analyse the optimisation problem~\eqref{ocp} by forming its dual~\eqref{dual-problem}. 
The optimal control (or primal) problem~\eqref{ocp} indeed rewrites
\[\inf_{u \in E, \; y(T) \in \overline{B}(\tilde y_T, \e)} F(u) = \inf_{u \in E} F(u) + G_\e(L_T u),\]
where $G_\e = \delta_{\overline{B}(\tilde y_T, \e)} \in\Gamma_0(X)$ conjugates to $G_\e^\ast : z \mapsto \e \|z\| +\langle \tilde y_T, z\rangle $. 

In Subsection~\ref{subsec-strong-duality}, we first establish through Proposition~\ref{compute-conj} that the functions $F$ is in $\Gamma_0(E)$ and provide the formula for its conjugate, showing that the above problem indeed admits a Fenchel-Rockafellar dual given by 
\[- \inf_{\adjtar \in X} F^\ast(L_T^\ast \adjtar) + G_\e^\ast(-\adjtar) = - \inf_{\adjtar \in X} J_{\e}(\adjtar).\]
That is, the optimisation problem~\eqref{dual-problem} is (up to a minus sign) the Fenchel-Rockafellar dual of the optimisation problem~\eqref{ocp}, and the corresponding weak duality is satisfied:  
\begin{equation}
\label{weak-duality}
\inf_{u \in E} F(u) + G_\e(L_T u) \geq - \inf_{\adjtar \in X} J_{\e}(\adjtar).
\end{equation}
In the same subsection, Lemma~\ref{strong-duality} verifies that a sufficient condition for the Fenchel-Rockafellar theorem to hold is met~\cite{Rockafellar1967}. Consequently, 
\begin{itemize}
    \item strong duality (\ie equality in~\eqref{weak-duality}) is satisfied,
    \item if the dual problem has a finite infimum, then the primal problem attains its infimum: in other words, there exists an optimal control. 
\end{itemize}
Subsection~\ref{existence-opt} is then devoted to the proof that under assumptions\eqref{cond-approx} (resp.~\eqref{cond-exact}), the dual problem~\eqref{dual-problem} has a finite infimum, showing that optimal controls exist. 

We also discuss whether the infimum of~\eqref{dual-problem} is attained, in which case we may speak of primal-dual optimal pairs. 
We recall that any primal-dual optimal pair $(u^\star, \adjtar^\star)$ is a saddle point of the Lagrangian  
\[(u,q) \in E \times X \mapsto \langle q, L_T u\rangle_X  + F(u) - \delta_{\overline{B}(\tilde y_T, \e)}^\ast(q) = \langle q, L_T u\rangle_X  + F(u) - \sigma_{\overline{B}(\tilde y_T, \e)}(q),\]
which in turn is equivalent to the first-order optimality conditions 
\begin{equation}
    \label{saddle-point}
u^\star \in \partial F(L_T^\ast \adjtar^\star), \qquad \adjtar^\star \in -\partial \delta_{\overline{B}(\tilde y_T, \e)}(L_T u^\star).
\end{equation}
In Subsection \ref{unique-extremal}, we analyse the uniqueness and extremality of optimal controls by studying these optimality conditions, thereby completing the proofs of Theorems~\ref{thm-main-approx} and \ref{thm-main-exact}.

\subsection{Necessary conditions for reachability}
\label{necessary-reach}

We start by proving the necessity of our two conditions~\eqref{cond-approx} and~\eqref{cond-exact} for approximate and exact reachability, respectively.
First, we prove Lemma~\ref{cont-cone}.

\noindent \textit{Proof}(of Lemma~\ref{cont-cone}).
Given the definition of $F^\ast$, we have
$F^\ast(L_T^\ast \adjtar) = 0$ if and only if $\sigma_{\U_r}(L_T^\ast \adjtar(t))= 0$ for a.e. $t \in (0,T)$.

By definition of~$\sigma_{\U_r}$, the condition $\sigma_{\U_r}(L_T^\ast \adjtar(t)) = 0$ is equivalent to $\scl{u}{L_T^\ast\adjtar(t)}_U \leq 0$ for all $u \in \U_r$, which in turn is equivalent to $\scl{u}{L_T^\ast\adjtar(t)}_U \leq 0$ for all $u \in P_r = \cone(\U_r)$ and hence to $L_T^\ast \adjtar(t) \in P_r^\circ$.

In other words, we have proved the equivalence 
\begin{equation*}
F^\ast(L_T^\ast \adjtar) = 0 \quad \iff \quad
\left(\text{for a.e. } t \in (0,T), \; L_T^\ast \adjtar(t) \in P_r^\circ\right) 
\end{equation*}
Hence,~\eqref{cond-approx} is equivalent to~\eqref{CNS-cone}.
\qed

\begin{prpstn}
\label{reach-implies-cond}
Let $P_r = \cone(\U_r)$.

If $\tilde y_T$ is approximately reachable from $y_0$ in time $T$ under the constraints $P_r$, then~\eqref{cond-approx} holds. 

If $\tilde y_T$ is exactly reachable from $y_0$ in time $T$ under the constraints $P_r$ with controls in $L^2_{\U_r}$, then~\eqref{cond-exact} holds.
\end{prpstn}

\begin{proof} 

\textit{Approximate reachability.}
Assume that $\tilde y_T$ is approximately reachable from $y_0$ in time $T$ under the constraints $P_r$. In order to prove~\eqref{cond-approx} we prove the equivalent condition~\eqref{CNS-cone}.

We let $P_{r,T} :=\{u \in E, \, \text{for a.e. } t \in (0,T), \; u(t) \in P_r\}$. By assumption, there exists $(u_\e) \in P_{r,T}$ such that $L_T u_\e$ converges strongly to $\tilde y_T$ in $X$, as $\e \to 0$. Hence for a given $\adjtar \in X$, we may pass to the limit $\e \rightarrow 0$ within the inequality $\langle L_T u_\e, \adjtar \rangle_X \leq \sup_{u \in P_{r,T}}\langle L_T u, \adjtar \rangle_X$, leading to
\[\langle \tilde y_T, \adjtar \rangle_X \leq \sup_{u \in P_{r,T}}\langle L_T u, \adjtar \rangle_X = \sup_{u \in P_{r,T}}\langle u, L_T^\ast \adjtar \rangle_E = \int_0^T \sup_{u \in P_{r,T}} \scl{u}{L_T^\ast \adjtar(t)}_U dt.\]
Assume that for a.e. $t \in (0,T)$, $L_T^\ast \adjtar(t) \in P_r^\circ$, which means that for a.e. $t \in (0,T)$ and all $u \in P_r$, there holds $\scl{u}{L_T^\ast \adjtar(t)}_U \leq 0$. In particular, we find that the right-hand side of the inequality must be nonpositive, hence so is the left-hand side, \ie $\langle \tilde y_T, \adjtar \rangle_X \leq 0$.

\textit{Exact reachability.}
Now assume that $\tilde y_T$ is exactly reachable from $y_0$ in time $T$ under the constraints~$P_r$ with controls in $L^2_{\U_r}$.
Then one can find $u \in  L^2_{\U_r}$ such that $L_T u = \tilde y_T$.
For a.e. $t \in (0,T)$, the inclusion $u(t) \in P_r = \cone(\U_r)$ shows that $u(t) \in j_{\U_r}(u(t)) \; \U_r$  by Lemma~\ref{basic-gauge}[(iv)] applied to the set $\U_r$. We may thus write $u(t) = j_{\U_r}(u(t)) \, v(t)$ with $v(t) \in \U_r$.
We now bound as follows
\begin{align*} 
\langle \tilde y_T, \adjtar \rangle_X & = 
\langle L_T u, \adjtar \rangle_X 
 = \langle u, L_T^\ast \adjtar \rangle_E 
 = \int_0^T \langle u(t), L_T^\ast \adjtar(t)\rangle_U \,dt  \\
& =\int_0^T  j_{\U_r}(u(t)) \, \langle v(t), L_T^\ast \adjtar(t)\rangle_U \,dt 
 \leq \int_0^T  j_{\U_r}(u(t)) \, \sigma_{\U_r}( L_T^\ast \adjtar(t)) \,dt
\end{align*}
By the Cauchy-Schwarz inequality, we find
\begin{align*} 
\langle \tilde y_T, \adjtar \rangle_X & \leq 
\Big(\int_0^T j_{\U_r}^2(u(t))\,dt\Big)^{1/2}  \Big(\int_0^T  \sigma_{\U_r}^2( L_T^\ast \adjtar(t)) \,dt\Big)^{1/2} \\
& = (2 F(u))^{1/2}  (2 F^\ast(L_T^\ast \adjtar))^{1/2},
\end{align*}
which is exactly the expected inequality~\eqref{cond-exact} with $c = 2 (F(u))^{1/2}<+\infty$ since $u \in L^2_{\U_r}$. 
\end{proof}

\subsection{Primal and dual problems}
\label{subsec-strong-duality}
We derive the dual problem and establish strong duality whatever the value of $\e \geq 0$.
\begin{prpstn}
\label{compute-conj}
The function $F$ defined by~\eqref{primal-cost} is in $\Gamma_0(E)$ and its Fenchel conjugates is given by formula~\eqref{dual-cost}.
\end{prpstn}
\begin{proof}
Thanks to Fenchel-Moreau's theorem, it is equivalent to prove that $F^\ast$ defined by~\eqref{dual-cost} is in $\Gamma_0(E)$ and conjugates to $F$.

Since $\tfrac {1}{2}\sigma^2_{\U_r} \in \Gamma_0(X)$ by Lemma~\ref{conj-supp-square}, one may use~\cite{Nous}[Lemma A.5.], to find that $F^\ast \in \Gamma_0(E)$, and further that conjugation and integration commute: for $u \in E$,
\[(F^\ast)^\ast(u) = \int_0^T \Big(\tfrac{1}{2} \sigma_{\U_r}^2\Big)^\ast(u(t)) \,dt = \int_0^T \tfrac{1}{2} j_{\U_r}^2(u(t)) \,dt = F(u).\]
where we again used Lemma~\ref{conj-supp-square}.

\end{proof}

We now turn to strong duality.

\begin{lmm}
\label{strong-duality}
The function $F^\ast$ is continuous at $0 = L_T^\ast 0$. 
\end{lmm}
\begin{proof}
    Since $\U$ is bounded, so is $\U_r$. Denoting $R>0$ a corresponding bound, the function $\sigma_{\U_r}$ is $R$-Lipschitz, hence for all $z \in U$  the inequality $|\sigma_{\U_r}(z)| \leq K \|z\|_U$.
It follows that for all $p\in E$
\[ 0 \leq F^\ast(p)\leq \tfrac{1}{2} K^2 \int_0^T \|p(t)\|_U^2\,dt = \tfrac{1}{2} K^2 \|p\|_E^2,\]
The continuity of $F^\ast$ at $0$ follows. \end{proof}

\begin{rmrk}
In fact, the above inequalities prove that the convex lsc function $F^\ast$ takes finite values on the whole of $E$, hence is continuous on $E$ and not merely at $0$.
\end{rmrk}

Consequently, the Fenchel-Rockafellar theorem applies: strong duality holds and furthermore, the infimum of the optimisation problem is attained if finite. 

\subsection{Existence of optimal controls}
\label{existence-opt}

In order to prove that the optimisation problem attains its infimum, we establish that the dual problem has a finite infimum. We thereby prove the converse to Proposition~\ref{reach-implies-cond}: conditions~\eqref{cond-approx} (resp.~\eqref{cond-approx}) are sufficient for approximate (resp. exact) reachability.

We start with approximate reachability, in which case the infimum is a minimum.
\begin{prpstn}
Assume that $\e>0$, and that~\eqref{cond-approx} holds. Then the dual problem has a minimum. In particular, under~\eqref{cond-approx}, $\tar$ is approximately reachable from $y_0$ in time $T>0$ under the constraints $P_r = \cone(\U_r)$.
\end{prpstn}
 
\begin{proof}
First, we notice that $J_{\e}$ is in $\Gamma_0(X)$: hence it suffices to prove that $J_{\e}$ is coercive in $X$ to conclude that it has a minimum. 

We show that coercivity holds following the proof of Proposition 3.5 in~\cite{Nous}, by proving that
\begin{equation*}
    \liminf_{\|\adjtar\|_X\to\infty} \frac{J_{\e}(\adjtar)}{\|\adjtar\|_X}>0. 
\end{equation*}
We take a sequence $\|\adjtar^n\|_X\to \infty$ and denote $\textstyle q_f^n := \tfrac{\adjtar^n}{\|\adjtar^n\|_X}$.
By homogeneity of the different terms involved, we have
\begin{equation*}
    \frac{J_{\e}(p^n_f)}{\|p^n_f\|_X}= \|p^n_f\|_X F^\ast(L_T^\ast q_f^n)-\left\langle \tilde y_T, q_f^n\right\rangle_{X} + \varepsilon
\end{equation*}
and hence if $\liminf_{n\to \infty} F^\ast(L_T^\ast q_f^n)>0$, then \begin{equation*} \liminf_{n\to \infty} \frac{J_{\e}(p^n_f)}{\|p^n_f\|_X}=+\infty.\end{equation*}
Let us now treat the remaining case where $\liminf_{n\to \infty} F^\ast(L_T^\ast q_f^n)= 0$. Since $\|q_f^n\|_X = 1$, upon extraction of a subsequence, we have $q_f^n \rightharpoonup q_f$ weakly in $X$ for some $q_f \in X$.
Since $L_T^\ast \in L(X,E)$, we have $L_T^\ast q_f^n \rightharpoonup L_T^\ast q_f$ weakly in $E$.

Now, since $F^\ast$ is convex and strongly lsc on $E$, it is (sequentially) weakly lsc and taking the limit we obtain $F^\ast(L_T^\ast q_f) = 0$.
From our assumption that~\eqref{cond-approx} holds, this leads to $\langle \tilde y_T, q_f\rangle_X \leq 0$.
Then we end up with
\begin{equation*}
   \liminf_{n\to \infty} \frac{J_{\e}(p^n_f)}{\|p^n_f\|_X}\geq  -\langle \tilde y_T, q_f \rangle_X+ \varepsilon  \geq \varepsilon >0.
\end{equation*}

\end{proof}

We next consider exact reachability, recalling the definition
\[c^\star = \sup_{\|q_f\|_X =1}\frac{\scl{\tilde y_T}{q_f}_X}{F^\ast(L_T^\ast q_f)^{1/2}} ,\]
with the convention $0/0 = 0$, $\alpha/0 = +\infty$ for $\alpha>0$.
Also recall that $c^\star \in (0,+\infty]$ if and only if $\tilde y_T  \neq 0$, and in the latter case, one may restrict the set of vectors $q_f \in X$ over which is performed, namely
\[c^\star = \sup_{\substack{\|q_f\|_X =1\\\scl{\tilde y_T}{q_f}_X > 0}} \frac{\scl{\tilde y_T}{q_f}_X}{F^\ast(L_T^\ast q_f)^{1/2}}.\]
\begin{prpstn}
Assume that $\e=0$, and that~\eqref{cond-exact} holds. Then the dual problem has a finite infimum. In particular, under~\eqref{cond-exact}, $\tar$ is approximately reachable from $y_0$ in time $T>0$ under the constraints $P_r = \cone(\U_r)$ with controls in $L^2_{\U_r}$.

Furthermore, the dual problem admits a minimiser if and only if $c_r^\star$ is attained. 
\end{prpstn}
\begin{proof}
    Let $r \in \{2,\infty\}$. 
When $\e=0$, the dual functional reads
\[ J_{0}(\adjtar) = F^\ast(L_T^\ast \adjtar) -\langle \tilde y_T, \adjtar \rangle_X.\] 
Thanks to the assumption that~\eqref{cond-exact} is satisfied, $J_{0}$ may be lower bounded as follows  \[\forall \adjtar \in X, \quad J_{0}(\adjtar) \geq F^\ast(L_T^\ast \adjtar) -c F^\ast(L_T^\ast \adjtar)^{1/2}. \]
Since the mapping $z \mapsto z - c z^{1/2}$ is lower bounded on $[0, +\infty)$, $\inf_{\adjtar \in X} J_{0}(\adjtar)> -\infty$, as wanted.

Now let us prove that the infimum is a minimum if and only if $c^\star$ is attained.
We rule out the case where $\tar = S_T y_0$, in which case $c^\star = 0$ is attained with $\adjtar^\star = 0$.

In the case where $\tar \neq S_T y_0$, we have $\inf J_{0} <0$ and $c^\star>0$. Indeed, if $\inf J_{0} =0$ were to hold, $0$ would be an optimal control meaning that $y_f = S_T y_0$. 
Let us look at the behaviour of $J_{0}$ over any possible half-line, using the equality 
\[\inf_{\adjtar \in X} J_{0}(\adjtar) = \inf_{\|q_f\|_X =1} \inf_{\lambda \geq 0} J_{0}(\lambda q_f) = : \inf_{\|q_f\|_X =1} m(q_f).\]
Since $\adjtar \mapsto F^\ast(L_T^\ast \adjtar)$ is 2-positively homogeneous, we find for a fixed $q_f \in X$ 
\[J_{0} (\lambda q_f) = F^\ast(L_T^\ast q_f) \lambda^2 - \scl{\tilde y_T}{q_f}_X \lambda.\]
If $\scl{\tilde y_T}{q_f}_X  \leq 0$, we clearly  have $m(q_f) = 0$ (recall that $F^\ast(L_T^\ast q_f) \geq  0$).
We are left with the case $\scl{\tilde y_T}{q_f}_X > 0$, which by~\eqref{cond-exact} imposes $F^\ast(L_T^\ast q_f) >0$. Then the quadratic function of $\lambda$ above is minimised uniquely at $\lambda^\star(q_f) := \tfrac{\scl{\tilde y_T}{q_f}_X}{2 F^\ast(L_T^\ast q_f)}$, leading to 
$m(q_f) =-\tfrac{\scl{\tilde y_T}{q_f}^2_X}{4 F^\ast(L_T^\ast q_f)}$.

In other words, we have found that 
\[\inf_{\adjtar \in X} J_{0}(\adjtar) = \inf_{\substack{\|q_f\|_X =1\\\scl{\tilde y_T}{q_f}_X > 0}}-\frac{\scl{\tilde y_T}{q_f}^2_X}{4 F^\ast(L_T^\ast q_f)} = -\frac{(c_r^\star)^2}{4}.\]

As a result, if $c^\star$ is attained by some $q_f \in X$ with $\|q_f\|_X =1$ and $\scl{\tilde y_T}{q_f}_X > 0$, then $\lambda^\star(q_f) q_f$ minimises $J_{0}$ over $X$.
Conversely, if $J_{0}$ has a minimiser $\adjtar^\star$, then $q_f^\star = \tfrac{\adjtar^\star}{\|\adjtar^\star\|_X}$ (which must satisfy $\langle \tilde y_T, q_f^\star\rangle_X > 0$), $c^\star$ is attained at $q_f^\star$.
\end{proof}

\subsection{Uniqueness, optimality conditions}
\label{unique-extremal}

\begin{prpstn}\label{prop-uniqueness}
Let $\e \geq 0$ and assume that $J_{\e}$ has a minimum and let $\adjtar^\star \in X$ be a minimiser. Then at least one control in $u \in  \partial F^\ast(L_T^\ast \adjtar^\star)$ is optimal.
Furthermore, 
\begin{itemize}
    \item  there exists a unique such control if and only if~\eqref{extremal} holds,
    \item if in addition~\eqref{extremality} holds, the unique control takes values in $P=\cone(\U)$.
\end{itemize}
\end{prpstn}

\begin{proof}
Since the dual value is finite, optimal controls do exist as already mentioned. Let $u^\star$ be such an optimal control. Then the pair $(u^\star, \adjtar^\star)$ is a saddle point for the Lagrangian. In particular, the first of the two optimality conditions~\eqref{saddle-point} holds, \ie $u^\star \in  \partial F^\ast(L_T^\ast \adjtar^\star)$.

Let us now discuss uniqueness. 
By Lemma~\cite{Nous}[Lemma A.5], the latter inclusion is equivalent to $u^\star(t) \in \tfrac{1}{2} \partial \sigma^2_{\U_r}(L_T^\ast \adjtar^\star(t))$ for a.e. $t \in (0,T)$, which rewrites \[u^\star(t) \in \sigma_{\U_r}(L_T^\ast \adjtar^\star(t)) \, \partial \sigma_{\U_r}(L_T^\ast \adjtar^\star(t)) = \sigma_{\U_r}(L_T^\ast \adjtar^\star(t)) \argmax_{v \in \U_r}\langle L_T^\ast \adjtar^\star(t), v\rangle.\] by the chain rule. That is, we have obtained~\eqref{opt_control}. 

These inclusions define a unique control if for a.e $t \in (0,T)$, we have either $L_T^\ast \adjtar^\star(t) \notin \sing(\U_r)$ by definition of the latter set, or if $\sigma_{\U_r}(L_T^\ast \adjtar^\star(t)) = 0$, which is equivalent to $L_T^\ast \adjtar^\star(t) \in \cone(\U_r)^\circ = P_r^\circ$. That is, we have proved that $u^\star \in \partial F^\ast(L_T^\ast \adjtar^\star)$ defines a unique optimal control if and only if~\eqref{extremal} is satisfied.


All is left to prove is that this (assumed to be unique) control $u^\star \in \partial F^\ast(L_T^\ast \adjtar^\star)$ takes values in $P = \cone(\U)$ and not merely in $P_r = \cone(\U_r)$, whenever we assume in addition that~\eqref{extremality} holds.

For a.e. $t \in (0,T)$, \[u(t) \in \sigma_{\U_r}(L_T^\ast \adjtar^\star(t)) \argmax_{v \in \U_r}\;  \langle  L_T^\ast \adjtar^\star(t), v\rangle,\]
and we have shown that, when~\eqref{extremal} holds, this inclusion defines a unique control.  
If $M(t):= \sigma_{\U_r}(L_T^\ast \adjtar^\star(t))=0$, there is nothing to prove since $0 \in \U$. Now if $M(t) \neq 0$, $\tfrac{u(t)}{M(t)}$ belongs to a set that is reduced to a singleton. Since the linear function $v\mapsto \langle L_T^\ast \adjtar^\star(t), v\rangle$, when maximised over the convex set $\U_r$, must have at least one maximiser that is an extremal point of $\U_r$, this singleton must then be an extremal point of~$\U_r$, and hence an element of $\U$ by~\eqref{extremality}.

\end{proof}

We end this subsection by discussing the uniqueness of the dual optimal variable in the case of approximate reachability $\e>0$.

\begin{prpstn}
Assume that $\tar$ is approximately reachable from $y_0$ in time $T>0$ under the constraints $P_r = \cone(\U_r)$.
Then whatever the value of $\e>0$, the dual minimiser $\adjtar^\star$ is unique (and $\adjtar^\star \neq 0$ if and only if $\|\tilde y_T\|_X > \e$).
\end{prpstn}
\begin{proof}
Let us first prove the uniqueness of $\adjtar^\star$, and the fact that it equals $0$ if and only if $\|\tilde y_T\|_X \leq \e$. The proof is similar to~\cite{Nous}[Proposition 3.10.], but for completeness we provide it in concise form below.

Consider any optimal control $u^\star$. By the second inclusion $\adjtar^\star \in -\partial \delta_{\overline{B}(\tilde y_T, \e)}(L_T u^\star)$ of~\eqref{saddle-point} we find that $L_T u^\star$ lies at the boundary of $\overline{B}(\tilde y_T, \e)$. 

If $\|\tilde y_T\|_X \leq \e$, then as already seen, $0$ is the unique optimal control \ie $\{L_T u^\star,\, u^\star\text{ is optimal}\}=\{0\}$.
If $\|\tilde y_T\|_X > \e$, since the set of minimisers of a convex function is convex, the set $\{L_T u^\star, \, u^\star \text{ is optimal}\}$ is a convex subset of the sphere $S(\tilde y_T, \e)$. The closed ball being strictly convex in the Hilbert space $X$ there exists some $y^\star \in \overline{B}(\tilde y_T, \e)$ with $\|y^\star - \tilde y_T\|_X = \e$ such that \begin{equation}\label{unique-target}
\{L_T u^\star, \, u^\star \text{ is optimal}\} = \{y^\star\}.\end{equation}
Thus, in any case, the set of targets reached by optimal controls is always reduced to a single point, which we can denote $y^\star$ in general.

We now return to the inclusion $\adjtar^\star \in - \partial \delta_{\overline{B}(\tilde y_T, \e)}(L_T u^\star)  =-\partial \delta_{\overline{B}(\tilde y_T, \e)}(y^\star)$.
First assume $\|\tilde y_T\|_X \leq \e$, then  $y^\star = 0$ and then $\adjtar^\star \in -\partial\delta_{\overline{B}(\tilde y_T, \e)}(0)$.
If $\|\tilde y_T\|_{X}<\e$, the latter set reduces to $\{0\}$, leading to $\adjtar^\star = 0$.
Otherwise, $0\in \partial \overline{B} (\tilde y_T, \e)$ and we find
\[\adjtar^\star \in\left\{\lambda \frac{\tilde y_T}{\e}, \lambda \geq 0\right\} = \{\lambda \tilde y_T, \lambda \geq 0\}.\]
Restricting the function $J_{\e}$ defining the dual problem to the above half-line, using the homogeneities of each of its terms, and the fact that $\|\tilde y_T\|_X=\e$, we get
\begin{equation*}\label{non-H-case-2}\gamma_0(\lambda):=J_{\e}(\lambda \tilde y_T)=a_0 \lambda^2, \quad \lambda\geq 0.\end{equation*}
It is clear that $0$ is the unique minimiser of $\gamma_0$. In other words, we have proved $\adjtar^\star = 0$ whenever $\|\tilde y_T\|_X \leq \e$.

Now assume $\|\tilde y_T\|_X > \e$. In this case, note that $\adjtar^\star \neq 0$, since otherwise the dual problem would admit the minimum $0$, hence the primal optimisation problem would also be of minimum $0$. Hence $0$ would be the (unique) optimal control, which would lead to $\|\tilde y_T\|_X \leq \e$, a contradiction. 

To prove the uniqueness of $\adjtar^\star \neq 0$ we argue as follows.
Since $y^\star$ lies at the boundary of $\overline{B}(\tilde y_T, \e)$, we find
\[\adjtar^\star \in\left\{\lambda \left(\frac{\tilde y_T-y^\star}{\e}\right), \lambda \geq 0\right\} = \{\lambda \left(\tilde y_T - y^\star\right), \lambda \geq 0\}.\]
Restricting $J_{\e}$ to the above half-line as previously, we find
\[\gamma(\lambda):=J_{\e}(\lambda (\tilde y_T-y^\star)) = a \lambda ^2 + b \lambda, \quad \lambda \geq 0,\]
where, using $\|\tilde y_T - y^\star\|_X = \e$ and the homogeneities involved $
a= F^\ast(L_T^\ast (\tilde y_T-y^\star))$ and $b=-\langle \tilde y_T, \tilde y_T-y^\star  \rangle_X + \e^2$.
By coercivity, $a >0$, and since $\adjtar^\star \neq 0$, we have $b<0$. 

Thus, $\gamma$ has a unique minimiser $\lambda^\star:=-b/2a>0$. Hence, $\adjtar^\star=\lambda^\star(\tilde y_T-y^\star)$,
and the dual optimal variable is unique.
\end{proof}


\paragraph{Acknowledgments.} All three authors acknowledge the support of the ANR project TRECOS, grant number ANR-20-CE40-0009.

\appendix 
\section{Further proofs and results}
\subsection{Elementary results about gauge functions}
We here gather several basic results used throughout the work. We do not claim any originality, but provide them here for completeness and readability. 

Throughout this subsection, we let $H$ be a Hilbert space and $C$ be a non-empty, bounded, closed and convex set containing $0$
\label{app-gauge-results}
\begin{lmm}
\label{basic-gauge}
For any $u \in H$, one has
\begin{itemize}
    \item[(i)] $j_C(u) = 0 \iff u = 0$,
    \item[(ii)] for any $\alpha>j_C(u)$, $u \in \alpha C$,
    \item[(iii)] for any $\alpha>0$, $j_C(u) \leq \alpha \, \iff \, u \in \alpha C$,
    \item[(iv)] $u \in \cone(C) \, \iff \, u \in j_C(u) C$.
\end{itemize}
\end{lmm}
\begin{proof}
    If $j_C(u) = 0$, then we may find $(\e_n)$ with $\e_n>0, \, \e_n \to 0$  and $(c_n) \in C^\N$ such that for all $n$, $u = \e_n c_n$. Hence $c_n = \e_n^{-1} u$, and since $C$ is bounded, this enforces $u=0$. The converse assumption is trivial since $0 \in C$, so (i) is proved.

    Since $\alpha> j_C(u)$, we may find $j_C(u) < \beta \leq \alpha$ such that $u \in \beta C$. Hence there exists $c \in C$ such that $u = \beta c = \alpha(\tfrac{\beta}{\alpha} c) = \alpha(\tfrac{\beta}{\alpha}c + (1-\tfrac{\beta}{\alpha}) 0)$. The vector $\tfrac{\beta}{\alpha}c$ appears as the convex combination of $c \in C$ and $0 \in C$, which shows that $u \in \alpha C$ and proves (ii).

    Let $\alpha>0$ and $u \in H$. If $u \in \alpha C$, then $j_C(u) \leq \alpha$ by definition. Conversely, assume $j_C(u) \leq \alpha$. We may find a sequence $(\alpha_n)$ with $\alpha_n > \alpha$ converging to $\alpha$. Since $\alpha_n > j_C(u)$, (ii) ensures the existence of $c_n \in C$ such that $u = \alpha_n c_n \iff c_n = \alpha_{n}^{-1} u$. Passing to the limit $n \to +\infty$ (since $\alpha>0$), we find that $(c_n)$ converges to $c:= \alpha^{-1} u$. By closedness of $C$, $c \in C$, showing that $u \in \alpha C$ and finishing the proof of~(iii).

    Let $u \in H$. If $u=0$, the equivalence is clear: $0 \in \cone(C)$ by the assumption $0 \in C$, and since $j_C(u)=0$, any $v \in C$ is such that $0 = u = j_C(u) v = 0$. Now assume that $u \neq 0$. By (i), we know that $j_C(u) \neq 0$, hence if $u = j_C(u) v$ with $v \in C$, we have $u \in \cone(C)$. Conversely, if $u \in \cone(C)$, we use (iii) with $\alpha = j_C(u)>0$, showing that $u \in \alpha C = j_C(u) C$, as wanted.%

\end{proof}

\begin{lmm}
\label{closed_generated_cone}
If $0$ is the interior of $C$ relative to the cone it generates, i.e., if
\[\exists \delta>0, \quad  \cone(C) \cap \overline{B}(0,\delta) \subset C,
\]
then $\cone(C)$ is closed.
\end{lmm}
\begin{proof}
Let $(p_n) \in \cone(C)^{\N}$ be a convergent subsequence, of limit $p \in H$.
By Lemma~\ref{basic-gauge}(iv), we may write $p_n = \lambda_n c_n$ with $(c_n) \in C^\N$ and $\lambda_n = j_C(p_n)$. If $p=0$, there is nothing to prove, so we may assume $p \neq 0$. Since $(c_n)$ is bounded by hypothesis, and $(p_n)$ converges to $p \neq 0$, $\liminf \lambda_n > 0$. Hence, if $(\lambda_n)$ is upper bounded, then upon extraction we may write $\lambda_n \to \lambda$ with $\lambda \neq 0$. In this case, we find that $(c_n)$ converges to $c :=\tfrac{p}{\lambda}$. Since $C$ is closed, we have $c \in C$ and hence $p = \lambda c \in \cone(C)$.

To conclude, we only need to show that $(\lambda_n)$ cannot have a diverging subsequence. By contradiction, assume that it is the case: upon extraction, we may assume that $\lambda_n \to +\infty$ as $n \to +\infty$.
Then we may form the sequence $w_n = \lambda_n^{-1/2}p_n = \lambda_n^{1/2} c_n$. This sequence satisfies $w_n \to 0$ as $n \to +\infty$ as well as $(w_n) \in \cone(C)^{\N}$. We also compute $j_C(w_n) = \lambda_n^{-1/2} j_C(p_n) = \lambda_n^{1/2} \to +\infty$ as $n \to +\infty$. Hence for $n$ large enough we have both $w_n\in \cone(C) \cap \overline{B}(0,\delta)$ and $w_n \notin C$, contradicting the assumption that $0$ is the interior of $C$ relative to $\cone(C)$.
\end{proof}

\begin{lmm}
\label{conj-supp-square}
There hold
$\tfrac{1}{2}\sigma^2_C \in \Gamma_0(H)$, $\tfrac{1}{2} j_C^2 \in \Gamma_0(H)$  and $(\tfrac{1}{2}\sigma^2_C)^\ast = \tfrac{1}{2} j_C^2$.
\end{lmm}
\begin{proof}
Since $0 \in C$, $\sigma_C \geq 0$. If $f \in \Gamma_0(H) \in H, \, f\geq 0$, then $f^2 \in \Gamma_0(H)$. Hence $\tfrac{1}{2}\sigma_C^2 \in \Gamma_0(H)$ and $\tfrac{1}{2} j_C^2 \in \Gamma_0(H)$.

Now let us prove the equality. For $x = 0$, it is readily checked since $j_C(0) = 0$ and $(\tfrac{1}{2}\sigma^2_C)^\ast(0) = -\tfrac{1}{2} \inf_{y \in H} \sigma^2_C(y) = 0$.

For $f \in \Gamma_0(H)$ and  $g \in \Gamma_0(\R)$ non-decreasing, we recall the composition formula
\begin{equation*}
(g\circ f)^\ast(x)= \inf_{\alpha \geq 0} \left(g^\ast(\alpha)+\alpha f^\ast\Big(\frac{x}{\alpha}\Big)\right),\end{equation*}
with the convention that for $\alpha = 0$, $0\, f^\ast(\tfrac{y}{0}) = \sigma_{\mathrm{dom}(f)}(y)$. Hence for $x \in H \setminus \{0\}$, one finds
\begin{align*}
\left(\tfrac{1}{2} \sigma_{C}^2\right)^\ast(x) & = \inf_{\alpha \geq 0}\left(\tfrac{1}{2} \alpha^2 + \alpha \delta_{C}\big(\frac{x}{\alpha}\big)\right) 
 = \inf_{\alpha > 0}\left(\tfrac{1}{2} \alpha^2 + \alpha \delta_{C}\big(\frac{x}{\alpha}\big)\right) \\
& = \tfrac{1}{2}\Big(\inf_{\alpha>0, x \in \alpha C}  \alpha\Big)^2 = \tfrac{1}{2} j_{C}^2(x),
\end{align*}
where we discarded $\alpha = 0$ since \[0 \delta_{C}\big(\frac{x}{0}\big)= \sigma_{\mathrm{dom}(\sigma_{C})}(x) = \sigma_H(x) =\delta_{\{0\}}(x) = +\infty,\]
using that $\mathrm{dom}(\sigma_{C}) = H$, by boundedness of $C$.
\end{proof}

\subsection{Proof of Proposition~\ref{suff-cond-extremal}}
\label{proof-avoid-set}


\noindent \textit{Proof}(of Proposition~\ref{suff-cond-extremal}).

Let $\adjtar \neq 0$, and assume by contradiction that the set of interest $\{t \geq 0, \, S_t^\ast \adjtar \in K\}$ has positive measure. We let $p(t) = S_t^\ast \adjtar$ and $I_0 = \{t \geq 0, \, p(t) \in \overline{K}\}$, which by assumption also has positive measure. We have $\langle b, p(t) \rangle = 0$ for all $t \in I_0$ and all $b \in K^\perp$.

The set $I_0$ is closed by virtue of the closedness of $\overline{K}$ and the regularity $p \in C([0,+\infty), X)$. It is a standard fact that since $I_0$ is closed and has positive measure, the set of its limit points $I_1$ is also closed, satisfies $I_1 \subset I_0$ and has the same measure as $I_0$. 
Let $b \in K^\perp \cap \mathcal{D}(A)$ be fixed. We shall now prove that $\langle A b, p(t) \rangle_X = 0$ for all $t \in I_1$.

By definition of $I_1$, for a given $t \in I_1$ one may find a sequence $(t_k)$ of elements of $I_0$, tending towards~$t$ and such that $t_k \neq t$. Since $t \in I_0$ and $t_k \in I_0$, we have $\langle b, p(t_k)\rangle_X = \langle S_{t_k}b, \adjtar \rangle_X = 0 $ as well as $\langle b, p(t)\rangle_X =\langle S_{t}b, \adjtar \rangle_X =  0$. As a result
\[ \left\langle \frac{S_{t_k}b - S_t b}{t_k-t},\adjtar \right\rangle_X= 0.\]
Passing to the limit thanks to the assumption $b \in \mathcal{D}(A)$, we find $0 = \langle S_t A b, \adjtar \rangle_X = \langle A b, p(t)\rangle_X$. Hence for all $t \in I_1$, we both have $p(t) \in \{b\}^\perp$ and $p(t) \in \{A b\}^\perp$. 

Repeating the argument by induction, we define a family of decreasing sets $(I_n)$ that all have the same measure (that of $I_0$). Letting $J := \cap_{n \in \N} I_n$, we have found a set of positive measure such that for all $b \in K^\perp$ and all $0 \leq j \leq \ell(b)$, $p(t) \in \{A^j b\}^\perp$. The second hypothesis (ii) then provides $p(t) = 0$ for $t \in J$. By the injectivity of $(S_t^\ast)$ given by the first hypothesis (i), this leads to $\adjtar = 0$ and we reach a contradiction. \qed

\subsection{A counterexample}
\label{app-counterexample}
For the sake of completeness, we here provide an explicit finite-dimensional example of a situation where the conic constraint set $P$ is closed, but the image cone $L_T P$ is not. 

\begin{prpstn}
Consider the control system
\[\left\{\begin{aligned}
    \dot{y}_1&= y_2, \\
    \dot{y}_2&= u, 
\end{aligned}\right.\]
with conic constraint set given by $P = \{u  \geq 0\}.$
Then for all $T>0$,
\[L_T P = \{(y_1,y_2)\in \R^2, \; 0<y_1\leq Ty_2\} \cup \{(0,0)\}. \]
\end{prpstn}
The cone $L_T P$ is not closed, whatever the value of $T>0$. For instance, any target of the form $\tar = (0,a)$ with $a>0$ is approximately but not exactly reachable under the constraints $P$ from $(0,0)$ in time $T>0$.
\begin{proof}

We first prove that $L_T P$ is included in the announced set. Let $\tar  = ((\tar)_1, (\tar)_2) \in L_T P$. 
Then, there exists $u \in L^2(0,T)$, $u \geq 0$ such that for all $0 \leq t \leq T$
\[y_1(t)=\int_0^t y_2(t) \, dt, \quad y_2(t)=\int_0^t u(t) \, dt.\] Given that $u \geq 0$, is clear that $y_1 \geq 0$ and $y_2 \geq 0$ at all times, and in particular $(\tar)_1 \geq 0, \; (\tar)_2 \geq 0$.

First assume that $(\tar)_1 = 0$, then we would find $y_2 = 0$ on $(0,T)$, and in particular $(\tar)_2 = 0$. In this case, we find $\tar =(0,0)$ (obtained only with the null control).

Now let us assume that $(\tar)_1 >0$; all is left to prove is the inequality $(\tar)_1 \leq T (\tar)_2$. Remark that the function $y_2$ satisfies $y_2(T) = (\tar)_2, \; \textstyle \int_0^T y_2(t) \, dt = (\tar)_1$, and is nondecreasing. In particular, we have $y_2(t) \leq (\tar)_2$ for all $0 \leq t \le T$. Hence
$(\tar)_1 = \textstyle \int_0^T y_2(t) \,dt \leq T y_2(T) = T (\tar)_2$.

Conversely, let $\tar  = ((\tar)_1, (\tar)_2) \in \{(y_1,y_2)\in \R^2, \; 0<y_1\leq Ty_2\} \cup \{(0,0)\}$. If $\tar = 0$, there is nothing to prove. The interesting case is again $0<(\tar)_1 \leq T (\tar)_2$. Let us build a $H^1(0,T)$ nondecreasing function $w$ (which will correspond to $y_2$) satisfying $w(0) = 0$, $w(T) = (\tar)_2$ and $\textstyle \int_0^T w(t) \, dt = (\tar)_1$. We may for instance take $w$ defined by $w(t) = \tfrac{(\tar)_2}{t_1} t$ if $t \leq t_1$ and $w(t) = (\tar)_2$ for $t_1 <t \leq T$, where $t_1$ is adjusted so that $\textstyle  \int_0^T w(t) \, dt = (\tar)_1$, \ie $t_1 = 2 \tfrac{(\tar)_2 -T (\tar)_1}{(\tar)_2}$, which does satisfy $0 \leq t_1 \leq T$ thanks to the assumption $(\tar)_1 \leq T (\tar)_2$. 

The chosen function is in $H^1(0,T)$ and we may hence set $u = \dot w \in L^2(0,T)$, which is a nonnegative function steering $(0,0)$ to $\tar$ since the corresponding trajectory $(y_1, y_2)$ satisfies $y_2(T) =\textstyle  \int_0^T u(t)\,dt = w(T) = (\tar)_2$ and $y_1(T) =\textstyle  \int_0^T y_2(t)\,dt = \int_0^T w(t) \,dt = (\tar)_1$.

\end{proof}

\bibliographystyle{plain}
\bibliography{bib.bib}

\begin{thebibliography}{10}

\bibitem{ahmed_finite-time_1985}
N.~U. Ahmed.
\newblock Finite-time null controllability for a class of linear evolution
  equations on a {Banach} space with control constraints.
\newblock {\em Journal of optimization Theory and Applications},
  47(2):129--158, 1985.
\newblock Publisher: Springer.

\bibitem{argyriou2012sparse}
Andreas Argyriou, Rina Foygel, and Nathan Srebro.
\newblock Sparse prediction with the $ k $-support norm.
\newblock {\em Advances in Neural Information Processing Systems}, 25, 2012.

\bibitem{bauschke-combettes}
HH~Bauschke and PL~Combettes.
\newblock Convex analysis and monotone operator theory in hilbert spaces, 2011.
\newblock {\em CMS books in mathematics}, 10:978--1.

\bibitem{Berrahmoune2014}
Larbi Berrahmoune.
\newblock A variational approach to constrained controllability for distributed
  systems.
\newblock {\em Journal of Mathematical Analysis and Applications},
  416(2):805--823, 2014.

\bibitem{Berrahmoune2019}
Larbi Berrahmoune.
\newblock Constrained null controllability for distributed systems and
  applications to hyperbolic-like equations.
\newblock {\em ESAIM: Control, Optimisation and Calculus of Variations}, 25:32,
  2019.

\bibitem{berrahmoune_variational_2020}
Larbi Berrahmoune.
\newblock A variational approach to constrained null controllability for the
  heat equation.
\newblock {\em European Journal of Control}, 52:42 -- 48, 2020.

\bibitem{Boyer2013}
Franck Boyer.
\newblock On the penalised {HUM} approach and its applications to the numerical
  approximation of null-controls for parabolic problems.
\newblock {\em ESAIM: Proc.}, 41:15--58, 2013.

\bibitem{brammer1972controllability}
Robert~F Brammer.
\newblock Controllability in linear autonomous systems with positive
  controllers.
\newblock {\em SIAM Journal on Control}, 10(2):339--353, 1972.

\bibitem{CoronBook}
J.-M. Coron.
\newblock {\em Control and nonlinearity}, volume 136 of {\em Mathematical
  Surveys and Monographs}.
\newblock American Mathematical Society, Providence, RI, 2007.

\bibitem{Ervedoza_2020}
Sylvain Ervedoza.
\newblock Control issues and linear projection constraints on the control and
  on the controlled trajectory.
\newblock {\em North-W. Eur. J. of Math.}, 6:165--197, 2020.

\bibitem{fattorini1966complete}
Hector~O Fattorini.
\newblock Some remarks on complete controllability.
\newblock {\em SIAM Journal on Control}, 4(4):686--694, 1966.

\bibitem{hiriart-poly-norms}
Manlio Gaudioso and Jean-Baptiste Hiriart-Urruty.
\newblock Deforming $\|\cdot\|_1$ into $\|\cdot\|_\infty$ via polyhedral norms:
  A pedestrian approach.
\newblock {\em SIAM Review}, 64(3):713--727, 2022.

\bibitem{gugat2008norm}
Martin Gugat and Gunter Leugering.
\newblock $l^\infty$-norm minimal control of the wave equation: on the weakness
  of the bang-bang principle.
\newblock {\em ESAIM: Control, Optimisation and Calculus of Variations},
  14(2):254--283, 2008.

\bibitem{hautus1970stabilization}
MLJ Hautus.
\newblock Stabilization controllability and observability of linear autonomous
  systems.
\newblock In {\em Indagationes mathematicae (proceedings)}, volume~73, pages
  448--455. North-Holland, 1970.

\bibitem{kalman1960}
Rudolf~Emil Kalman et~al.
\newblock Contributions to the theory of optimal control.
\newblock {\em Bol. soc. mat. mexicana}, 5(2):102--119, 1960.

\bibitem{Kunisch-Wang-2013}
Karl Kunisch and Lijuan Wang.
\newblock Time optimal control of the heat equation with pointwise control
  constraints.
\newblock {\em ESAIM: Control, Optimisation and Calculus of Variations},
  19(2):460--485, 2013.

\bibitem{lions-1988}
Jacques-Louis Lions.
\newblock Exact controllability, stabilization and perturbations for
  distributed systems.
\newblock {\em SIAM review}, 30(1):1--68, 1988.

\bibitem{Lions-1992}
Jacques-Louis Lions.
\newblock Remarks on approximate controllability.
\newblock {\em Journal d’Analyse Math{\'e}matique}, 59(1):103, 1992.

\bibitem{Loheac2017}
J{\'e}r{\^o}me Loh{\'e}ac, Emmanuel Tr{\'e}lat, and Enrique Zuazua.
\newblock Minimal controllability time for the heat equation under unilateral
  state or control constraints.
\newblock {\em Mathematical Models and Methods in Applied Sciences},
  27(09):1587--1644, 2017.

\bibitem{Loheac2021}
J{\'e}r{\^o}me Loh{\'e}ac, Emmanuel Tr{\'e}lat, and Enrique Zuazua.
\newblock Nonnegative control of finite-dimensional linear systems.
\newblock {\em Annales de l'Institut Henri Poincar{\'e} C, Analyse non
  lin{\'e}aire}, 38(2):301--346, 2021.

\bibitem{Nous}
Camille Pouchol, Emmanuel Tr{\'e}lat, and Christophe Zhang.
\newblock Approximate control of parabolic equations with on-off shape controls
  by fenchel duality.
\newblock {\em To appear in Annales de l'Institut Henri Poincaré C, Analyse
  non linéaire}, 2024.

\bibitem{Rockafellar1967}
Ralph Rockafellar.
\newblock Duality and stability in extremum problems involving convex
  functions.
\newblock {\em Pacific Journal of Mathematics}, 21(1):167--187, 1967.

\bibitem{Rudin}
W.~Rudin.
\newblock {\em Functional Analysis}.
\newblock International series in pure and applied mathematics. McGraw-Hill,
  1991.

\bibitem{tucsnak2009observation}
Marius Tucsnak and George Weiss.
\newblock {\em Observation and control for operator semigroups}.
\newblock Springer Science \& Business Media, 2009.

\bibitem{zuazua2010switching}
Enrique Zuazua.
\newblock Switching control.
\newblock {\em Journal of the European Mathematical Society}, 13(1):85--117,
  2010.

\end{thebibliography}
\end{document}